\documentclass[10pt]{amsart}

\usepackage{tikz}
\usepackage{mathrsfs}
\usepackage{amsthm}

\usepackage{amsmath,amssymb}
\usepackage[all,poly,knot]{xy}
\usepackage{hyperref}
\usepackage[center]{titlesec} 
\usepackage{amsfonts}
\usepackage{color}
\titleformat{\section}{\centering\Large\bfseries}{\arabic{section}.}{1em}{}
\titleformat{\subsection}{\centering\large\bfseries}{\arabic{section}.\arabic{subsection}.}{1em}{}

\newtheorem{thm}{Theorem}[section]

\newtheorem{lem}[thm]{{Lemma}}

\newtheorem{prop}[thm]{Proposition}

\newtheorem{exa}[thm]{Example}
\theoremstyle{definition}
\newtheorem{setup}{Setup}
\newtheorem*{acknowledgement*}{Acknowledgements}

\theoremstyle{remark}
\newtheorem*{rmk}{Remark}
\newtheorem{nota}[thm]{Notation}

\numberwithin{equation}{section}

\def\Spec{\mathrm{Spec}} 
\def\Hom{\mathrm{Hom}} 
\def\mHom{\mathscr{H}\!om} 
\def\Fil{\mathrm{Fil}} 
\def\Gr{\mathrm{Gr}}

\def\End{\mathrm{End}} 
\def\DR{\mathrm{DR}} 
\def\mEnd{\mathscr{E}\!nd}
\def\w{\widetilde}

\def\oln{{\overline{\nabla}}}

\usepackage[top=1.5in, bottom=1.5in, left=1in, right=1in]{geometry}

\def\red{\textcolor[rgb]{1,0,0}}

\begin{document}
\title{Deformation theory of periodic Higgs-de Rham flows}
\author{Raju Krishnamoorthy}
\email{raju@uga.edu}  
\address{Department of Mathematics, University of Georgia, Athens, GA 30605, USA}
\author{Jinbang Yang}
\email{yjb@mail.ustc.edu.cn}
\address{Institut f\"ur Mathematik, Universit\"at Mainz, Mainz 55099, Germany}
\author{Kang Zuo}
\email{zuok@uni-mainz.de}
\address{Institut f\"ur Mathematik, Universit\"at Mainz, Mainz 55099, Germany}

\maketitle
\begin{abstract}
		In this note we study the deformation theory of periodic (logarithmic) Higgs-de Rham flows. Under suitable numerical assumptions, this is equivalent to the deformation theory of torsion (logarithmic) Fontaine-Faltings modules. As an application, we formulate an \emph{ordinarity} condition, which provides a sufficient condition for a $p^n$-torsion crystalline representation to deform to a $p^{n+1}$-torsion crystalline representation.
\end{abstract}
	
\section{Introduction}
	Let $k$ be a perfect field of odd characteristic $p>0$, let $W:=W(k)$ denote the ring of Witt vectors, and let $K:=\text{Frac}(W)$ be the field of fractions. Let $X/W$ be a smooth scheme and let $D\subset X$ to be a relative simple normal crossings divisor; the pair $(X,D)$ is a called \emph{smooth pair over $W$}. For any positive integer $n>0$, denote by $(X_n,D_n)$ the reduction of $(X,D)$ modulo $p^n$.  In this article, we typically consider objects on $X$ that are allowed logarithmic poles along $D$.

	Suppose $X/W$ is projective and consider an $1$-periodic (logarithmic) Higgs-de Rham flow $HDF_{X_n}$ over $X_n$, for an positive integer $n>0$,
	\begin{equation} \tiny
		\xymatrix{  
			& (V_{X_n},\nabla_{X_n}, \mathrm{Fil}_{X_n}) \ar[dr]^{\text{Gr}_{\mathrm{Fil}_{X_n}}} & \\
			(E_{X_n},\theta_{X_n})_0 \ar[ur]^{C^{-1}}
			&& (E_{X_n},\theta_{X_n})_1 \ar@/^20pt/[ll]^{\varphi_n}\\
		}
	\end{equation} 
	
Our goal is to understand the conditions/obstructions to lifting $HDF_{X_n}$ to a $1$-periodic Higgs-de Rham flow $HDF_{X_{n+1}}$. We now state our main results and applications.
	 
\begin{thm}[Theorem~\ref{uniqueHDF}] Let $(X,D)/W$ be a smooth pair with $X/W$ projective. Let $HDF_{X_n}$ be a periodic Higgs-de Rham flow over $(X_n,D_n)$. For a given lifting $(E,\theta)_{X_{n+1}}$ of the initial graded Higgs bundle over $(X_{n+1},D_{n+1})$, there is at most one Higgs-de Rham flow with initial term  $(E,\theta)_{X_{n+1}}$ that lifts $HDF_{X_n}$ up to isomorphism. 
\end{thm}

The following theorem gives a precise condition for the following. Let $(X_n,D_n)\subset (X_{n+1},D_{n+1})$ be a thickening of smooth pairs over $W_n(k)\subset W_{n+1}(k)$, with $X_n/W_n$ projective. Suppose we have a logarithmic Higgs bundle $(E_n,\theta_n)$ on $(X_n,D_n)$ that initiates a 1-periodic Higgs-de Rham flow. When does there exist a lift $(\w E_{n+1},\w{\theta}_{n+1})$ of $(E_n,\theta_n)$ together with a lift $(X_{n+2},D_{n+2})$ of $(X_{n+1},D_{n+1})$ over $W_{n+2}(k)$ such that $(\w E_{n+1},\w{\theta}_{n+1})$ initiates a 1-periodic Higgs-de Rham flow with respect to the thickening $(X_{n+1},D_{n+1})\subset (X_{n+2},D_{n+2})$? In the following, $\mathbb{K}$ and $H$ are explicitly defined in \ref{subsection:ordinary}, and $IC$ stands for the \emph{parametrized} inverse Cartier transform.
\begin{thm}[Theorem~\ref{main_thm}] Let $(X_n,D_n)/W_n$ be a smooth pair with $X_n/W_n$ projective. Suppose $\mathbb K$ is not empty and the projection $H\rightarrow \mathbb H^1\big(\mathrm{Gr}^0 \DR(\mEnd(E_1,\theta_1))\big)$ is surjective. Then there exists some finite extension $k'/k$ and $((\w E_{n+1},\w \theta_{n+1}),(X_{n+2},D_{n+2})) \in \mathbb K_{k'}$ such that
	\[\Gr\circ \mathrm{IC}((\w E_{n+1},\w \theta_{n+1}),(X_{n+2},D_{n+2})) = (\w E_{n+1},\w \theta_{n+1}).\]
	In other words, $(\w E_{n+1},\w \theta_{n+1})$ is a $1$-periodic Higgs bundle under the lifting $(X_{n+1},D_{n+1})\subset (X_{n+2},D_{n+2})$.	
\end{thm}

We briefly explain the main ingredients to the proofs of the main results. 
	
	Suppose there exists a graded Higgs bundle $(E_{X_{n+1}},\theta_{X_{n+1}})_0$ which lifts $(E_{X_n},\theta_{X_n})_0$. Taking the $n$-truncated inverse Cartier transform, one gets a de Rham bundle 
	\[(V_{X_{n+1}},\nabla_{X_{n+1}}) = C^{-1}\big((E_{X_{n+1}},\theta_{X_{n+1}})_0, (V_{X_n},\nabla_{X_n}, \mathrm{Fil}_{X_n}),\varphi_n\big).\]
	Suppose there exists a Hodge filtration $\mathrm{Fil}_{X_{n+1}}$ on $(V_{X_{n+1}},\nabla_{X_{n+1}})$. Then the associated grading Higgs bundle
	\[(E_{X_{n+1}},\theta_{X_{n+1}})_1 = \mathrm{Gr}(V_{X_{n+1}},\nabla_{X_{n+1}},\mathrm{Fil}_{X_{n+1}})\]
	lifts $(E_{X_n},\theta_{X_n})_1$. Via the isomorphism $\varphi_n$, both $(E_{X_{n+1}},\theta_{X_{n+1}})_1$ and $(E_{X_{n+1}},\theta_{X_{n+1}})_0$ lift $(E_{X_n},\theta_{X_n})_0$; therefore, their difference defines an obstruction to lifting the periodicity map $\varphi_n$.
	
	In conclusion, there are three obstructions to lifting a 1-periodic Higgs-de Rham flow:
	\begin{enumerate}
	\item the obstruction to lift the initial Higgs bundle;
	\item the obstruction to lift the Hodge filtration; and
	\item the obstruction to lift the periodicity map $\varphi_n$.
\end{enumerate}

We will study each of these in turn, using pure deformation theory. These deformation questions are general and have nothing to do with the ring of Witt vectors; we state our assumptions and results below.

\begin{setup}\label{setup}Let $S$ be a noetherian scheme, let $S\hookrightarrow \w{S}$ be a square-zero thickening, i.e., the ideal sheaf $\mathfrak a$ of $S$ satisfies $\mathfrak a^2=0$. Let $(\w{X},\w{D})$ be a smooth pair over $\w S$; that is, $\w{X}$ is a smooth $\w{S}$-scheme and $\w{D}\subset \w{X}$ is a $\w{S}$-flat relative simple normal crossings divisor with $D\times_{S}\w S$ reduced.  Set $(X,D): = (\w{X}\times_{\w{S}} S, \w{D}\times_{\w{S}} S)$. Then $\mathcal I =\mathcal O_{\w X} \cdot\mathfrak a$ is the ideal of definition of $X\hookrightarrow \w X$, and $\mathcal I^2=0$.

Set $\Omega^1_{\w X/\w S}$ to be the sheaf of relative 1-forms and $\Omega^1_{\w X/\w S}(\log \widetilde D)$ to be the sheaf of relative 1-forms with logarithmic poles along $\widetilde D$. (This latter sheaf  $\Omega^1_{\w X/\w S}(\log \widetilde D)$ may be also defined as the sheaf of 1-forms on the f.s. log scheme induced from the pair $(\w X,\w S)$) Both of these sheaves are finite and locally free. \end{setup}

In the context of Setup \ref{setup}, we show the following results on lifting (filtered) de Rham bundles and graded Higgs bundles. As explained above these results will be used to characterize when a Higgs-de Rham flow modulo $p^n$ lifts to a Higgs-de Rham flow modulo $p^{n+1}$. Note that Setup \ref{setup} does not assume that $X/S$ is projective.
\begin{thm}[Theorem~\ref{deformation_(V,nabla,Fil)}] Notation as in Setup \ref{setup} and let $(V,\nabla,\Fil)$ be a (logarithmic) filtered de Rham bundle over $(X,D)/S$. Then
	\begin{itemize}
		\item[(1).] the obstruction to lifting $(V,\nabla,\Fil)$ to a filtered de Rham bundle over $(\w X,\w D)/\w S$ lies in $\mathbb H^2\big(\Fil^0 \DR(\mEnd(V,\nabla)_\mathcal I)\big)$;
		\item[(2).] if $(V,\nabla,\Fil)$ has lifting $(\widetilde{V},\widetilde{\nabla},\w\Fil)$, then the set of liftings is a torsor for $\mathbb H^1\big(\Fil^0 \DR(\mEnd(V,\nabla)_\mathcal I)\big)$;
		\item[(3).] the infinitesimal automorphism group of $(\widetilde{V},\widetilde{\nabla},\w\Fil)$ over $(V,\nabla,\Fil)$ is $\mathbb H^0\big(\Fil^0 \DR(\mEnd(V,\nabla)_\mathcal I)\big)$.
	\end{itemize}
\end{thm}  
	
\begin{thm}[Theorem~\ref{deformation_(Fil)}] Notation as in Setup \ref{setup} and let $(V,\nabla,\Fil)$ be a filtered de Rham bundle over $(X,D)/S$. Let $(\w V,\w\nabla)$ be lifting of the underlying de Rham bundle $(V,\nabla)$ over $(\w X,\w D)/\w S$. Then
	\begin{itemize}
		\item[1).] the obstruction to lifting $\Fil$ to a Hodge filtration on $(\w V,\w\nabla)$ lies in $\mathbb H^1(\mathscr C)$;
		\item[2).] if $\Fil$ has a lifting, then the set of liftings is a torsor for $\mathbb H^0(\mathscr C)$.
	\end{itemize}
\end{thm}

\begin{thm}[Theorem~\ref{deformation_(E,theta,Gr)}] Notation as in Setup \ref{setup} and let $(E,\theta,\Gr)$ be a graded Higgs bundle over $(X,D)/S$. Then
	\begin{itemize}
		\item[1).] the obstruction to lifting $(E,\theta,\Gr)$ to a graded Higgs bundle over $(\w X,\w D)/\w S$
		lies in $\mathbb H^2\big(\mathrm{Gr}^0 \DR(\mEnd(E,\theta))_{\mathcal I}\big)$;
		\item[2).] if $(E,\theta,\Gr)$ has a graded lifting $(\widetilde{E},\widetilde{\theta},\w\Gr)$, then the set of liftings is a torsor for $\mathbb H^1\big(\mathrm{Gr}^0 \DR(\mEnd(E,\theta))_{\mathcal I}\big)$;
		\item[3).] the infinitesimal automorphism group of $(\widetilde{E},\widetilde{\theta},\w\Gr)$ over $(E,\theta,\Gr)$ is $\mathbb H^0\big(\mathrm{Gr}^0 \DR(\mEnd(E,\theta))_{\mathcal I}\big)$.
	\end{itemize}
\end{thm}

\begin{thm}[Theorem~\ref{thm: corollary of E1 degenarate}]
	Notation as in Setup \ref{setup} and let $(V,\nabla,\Fil)$ be a filtered de Rham bundle on $(X,D)/S$. Suppose that the Hodge-de Rham spectral sequence (\ref{spectral sequence}) attached to $(V,\nabla,\Fil)$ degenerates at $E_1$.
	\begin{itemize}
		\item[(1).] If $(V,\nabla)$ is liftable, then $(V,\nabla,\Fil)$ is also liftable.
		\item[(2).] Let $(\w V,\w \nabla)$ be a lifting of the underlying de Rham bundle of $(V,\nabla,\Fil)$. Then for two liftings $\w\Fil$ and $\w\Fil'$ of the Hodge filtration $\Fil$ there exists an isomorphism
		\[(\w V,\w \nabla,\w \Fil)\overset{\simeq}{\longrightarrow} (\w V,\w \nabla,\w \Fil')\]
		of filtered de Rham bundles on $(\w X,\w D)/\w S$.
	\end{itemize}
\end{thm}      

\begin{acknowledgement*}
We thank Ruiran Sun for many productive discussions that helped formulate many of the results in this note. R.K  gratefully acknowledges support from NSF Grant No. DMS-1344994 (RTG in Algebra, Algebraic Geometry and
Number Theory at the University of Georgia)
\end{acknowledgement*}

\section{Preliminaries}

\subsection{Filtered (logarithmic) de Rham bundles}
\paragraph{\emph{Filtered vector bundles.}}
Let $X$ be a scheme and let $V$ be a vector bundle on $X$. Recall that a \emph{descending filtration} of $V$ is a decreasing sequence of subbundles
\[V \supset \cdots \supset \Fil^{-2} V\supset \Fil^{-1} V\supset \Fil^{0} V\supset \Fil^{1} V\supset \Fil^{2} V\supset\cdots.\] 
The condition that $\Fil^{i}\subset \Fil^{i-1}$ is a subbundle means that it is locally a direct summand. This filtration is called separated and exhaustive if for sufficiently large integer $N>>0$ 
\[\Fil^{-N}V = V \quad \text{and} \quad \Fil^{N+1}V = 0.\]
In this note, all filtrations are assume to be descending, separated and exhaustive. The pair $(V,\Fil)$ is called a \emph{filtered vector bundle over $X$}.

Let $(V,\Fil)$ be a filtered vector bundle with $\Fil^\ell V$ free over $\mathcal O_X$ for each $\ell\in \mathbb Z$. Then a basis $\{e_1,e_2\cdots,e_r\}$ of $V$ is called \emph{adapted} basis of $(V,\Fil)$ if $\Fil^\ell V$ is generated by a subset of $\{e_1,e_2\cdots,e_r\}$ for each $\ell\in \mathbb Z$.

\begin{lem}\label{local freeness_(V,Fil)}
	Let $(V,\Fil)$ be a filtered vector bundle on a scheme $X$. Then there exists an open affine covering $\{U_i\}_{i\in I}$ of $X$ such that $\Fil^\ell V|_{U_i}$ is free $\mathcal O_{U_i}$-module for each $\ell\in \mathbb Z$ and each $i\in I$.	
\end{lem} 
\begin{proof}
	There are only finitely many vector bundles $\Fil^{\ell}$, so we may find an open cover that trivializes them all.
\end{proof}

\paragraph{\emph{Dual filtration.}} Let $V$ be a vector bundle on a scheme $X$. The dual vector bundle $V^\vee$ is defined by 
\[V^\vee(U):=\Hom_{\mathcal O_X(U)}(V(U),\mathcal O_X(U)),\quad \text{ for any open subset $U$ of $X$.}\]

Let $(V,\Fil)$ be a filtered vector bundle over a scheme $X$. The \emph{dual filtration} $\Fil^{\vee}$ on the dual vector bundle $V^\vee$ is defined by
\begin{equation}\label{dual_Fil}
\Fil^{\vee,\ell} (V^\vee):= \left(V/\Fil^{1-\ell}V\right)^\vee \subset V^\vee,
\end{equation} 
\paragraph{\emph{Filtration on $\Hom$.}}
Let $(V_1,\Fil_1)$ and $(V_2,\Fil_2)$ be two filtered vector bundle over $X$. There is a natural filtration $\Fil_{(\Fil_1,\Fil_2)}$ on the Hom vector bundle $\mHom(V_1,V_2)=V_1^\vee\otimes V_2$ given by
\begin{equation}\label{Hom_Fil}
\Fil^\ell_{(\Fil_1,\Fil_{2})}\mHom(V_1,V_2)=\sum_{\ell_1}(V_1/ \Fil_1^{\ell_1}V_1)^\vee \otimes\Fil_{2}^{\ell_1+\ell-1} V_2.
\end{equation}
Recall that $\mathrm{Hom}(V_1,V_2):=\Gamma(X,\mHom(V_1,V_2))$. The following lemma sum up some properties about the filtration on Hom vector bundles
\begin{lem}\label{main_lemma} Let  $(V_1,\Fil_1)$, $(V_2,\Fil_2)$ and $(V_3,\Fil_3)$ be three filtered vector bundles over a scheme $X$.
	\begin{itemize}
		\item [(1).] For any $f \in \mathrm{Hom}(V_1,V_2)$. Then
		$f\in \Fil^\ell_{(\Fil_1,\Fil_2)} \mathrm{Hom}(V_1,V_2)$ if and only if 
		$f(\Fil_1^{\ell'}V_1)\subset \Fil_{2}^{\ell'+\ell}V_2$ holds for all $\ell'$. 
		\item [(2).] The composition operator induces a map $\circ: \mathrm{Hom}(V_2,V_3) \times \mathrm{Hom}(V_1,V_2) \rightarrow \mathrm{Hom}(V_1,V_3)$.
		Under this composition map, the image of $\Fil^{\ell_1}_{(\Fil_2,\Fil_3)} \mathrm{Hom}(V_2,V_3) \times  \Fil^{\ell_2}_{(\Fil_1,\Fil_{2})} \mathrm{Hom}(V_1,V_2)$
		is contained in $\Fil^{\ell_1+\ell_2}_{(\Fil_1,\Fil_{3})}\mathrm{Hom}(V_1,V_3)$.
		\item [(3).] Assume $V_1=V_2=:V$ and $\mathrm{rank}(\Fil_1^\ell V)=\mathrm{rank}(\Fil_2^\ell V)$ for all $\ell$. Then
		\[\Fil_1=\Fil_2 \text{ if and only if } \mathrm{id}_{V} \in \Fil^0_{(\Fil_1,\Fil_{2})} \mathrm{End}(V).\]
	\end{itemize}
\end{lem}
\begin{proof} The first statement can be easily checked locally by choosing adapted local basis of $V^\vee_1$ and $V_2$.
	The second statement follows the first one. For the third one, it follows the fact that if $V'' \subseteq V'\subseteq V$ are sub bundles with $\mathrm{rank} V'=\mathrm{rank} V''$, then $V''=V'$.
\end{proof}

\paragraph{\emph{de Rham bundle.}}
Let $X$ be an $S$-scheme with relative flat normal crossing divisor $D$. A (logarithmic) connection on a vector bundle $V$ over $(X,D)/S$ is an $\mathcal O_S$-linear map $\nabla\colon V\rightarrow V\otimes_{\mathcal O_X} \Omega_{X/S}^1(\log D)$ satisfying the Leibniz rule $\nabla(rv) = v\otimes \mathrm{d}r + r\nabla(v)$ for any local section $r\in \mathcal O_X$ and $v\in V$. Give a connection, there are canonical maps 
\[\nabla\colon V\otimes\Omega_{X/S}^i(\log D) \rightarrow V\otimes\Omega_{X/S}^i(\log D)\]
given by $s\otimes \omega \mapsto \nabla(s)\wedge \omega +s \otimes \mathrm{d}\omega$.
The curvature $\nabla\circ \nabla$ is $\mathcal O_X$-linear and contained in $\mEnd(V)\otimes_{\mathcal O_X} \Omega^2_{X/S}(\log D)$. The connection is called \emph{integrable} and $(V,\nabla)$ is called a \emph{(logarithmic) de Rham bundle} if the curvature vanishes. For a de Rham bundle, one has a natural \emph{de Rham complex}:
\[\DR(V,\nabla):\qquad 0\rightarrow  V 
\xrightarrow{\nabla} V\otimes\Omega^1_{X/S}(\log D)
\xrightarrow{\nabla} V\otimes\Omega^2_{X/S}(\log D)
\xrightarrow{\nabla} V\otimes\Omega^3_{X/S}(\log D) \rightarrow \cdots \]
The hypercohomology of $\DR(V,\nabla)$ are called the \emph{de Rham cohomology} of the de Rham bundle $(V,\nabla)$ and are denoted as follows
\[H^i_{dR}(X,(V,\nabla)) := \mathbb H^i(X,\DR(V,\nabla)).\]

\paragraph{\emph{Connection on Hom vector bundles.}}
Let $(V_1,\nabla_1)$ and $(V_2,\nabla_2)$ be two vector bundles with a connections over $X/S$. Then there is a natural connection $\nabla_{(\nabla_1,\nabla_2)}$ on the vector bundle $\mHom_{\mathcal O_X}(V_1,V_2)$ given by
\begin{equation}\label{eqn:hom_connection}\nabla_{(\nabla_1,\nabla_2)}(f) := \nabla_2\circ f-f\circ \nabla_1
\end{equation}
for any local section $f\in \mHom(V_1,V_2)$. Let's note that 
\begin{itemize}
	\item for any $f\in \mHom(V_1,V_2)$, $\nabla_{(\nabla_1,\nabla_2)}(f)=0$ if and only if $f\in \mHom((V_1,\nabla_1),(V_2,\nabla_2))$, i.e. $f$ is parallel;
	\item  for any $\omega\in \mHom(V_1,V_2)\otimes\Omega^i_{X/S}(\log D)$,
	\[\nabla_{(\nabla_1,\nabla_2)}(\omega) = \nabla_2\circ \omega +(-1)^{i+1} \omega \circ \nabla_1.\]
\end{itemize}
\begin{lem}\label{Hom_(V,nabla)} Let $(V_1,\nabla_1)$ and $(V_2,\nabla_2)$ be two de Rham bundles over $X/S$. Then  
	\[\mHom((V_1,\nabla_1),(V_2,\nabla_2)) := \left(\mHom(V_1,V_2),\nabla_{(\nabla_1,\nabla_2)}\right)\]
	is also a de Rham bundle. In particular, $\nabla^{\End}_{\nabla_1}:=\nabla_{(\nabla_1,\nabla_1)}$ is an flat connection on $\mEnd(V_1)$.
\end{lem} 
\begin{proof}
Note that $\left(\nabla_{(\nabla_1,\nabla_2)}\right)^2(f) =(\nabla_2^2)\circ f - f\circ (\nabla_1^2)$. The lemma follows.
\end{proof} 
\paragraph{\emph{Filtered de Rham bundle.}} 
Let $X$ be an $S$-scheme with relative flat normal crossing divisor $D$. Let $(V,\nabla,\Fil)$ be a triple such that $(V,\Fil)$ is a filtered vector bundle and $(V,\nabla)$ be a de Rham bundle over $(X,D)/S$. Then the triple $(V,\nabla,\Fil)$ is called \emph{filtered de Rham bundle} if $\nabla$ and $\Fil$ satisfy the Griffith transversality. i.e. 
\[\nabla(\Fil^\ell V) \subset \Fil^{\ell -1}V\otimes \Omega_{X/S}^1(\log D).\]
Thus we have sub complexes of $\DR(V,\nabla)$
\[\Fil^\ell \DR(V,\nabla): 
\qquad 0
\rightarrow           \Fil^\ell V 
\xrightarrow{\nabla}  \Fil^{\ell-1} V \otimes \Omega^1_{X/S}(\log D)
\xrightarrow{\nabla}  \Fil^{\ell-2} V \otimes \Omega^2_{X/S}(\log D)
\xrightarrow{\nabla}  \cdots \]
These sub complexes define a filtration on $\DR(V,\nabla)$
\begin{equation}
\cdots 
\subset \Fil^{\ell +1} \DR(V,\nabla)
\subset \Fil^{\ell   } \DR(V,\nabla)
\subset \Fil^{\ell -1} \DR(V,\nabla)
\subset \cdots \subset \DR(V,\nabla)
\end{equation}
\begin{lem} \label{Hom_(V,nabla,Fil)}
	Let $(V_1,\nabla_1,\Fil_1)$ and $(V_2,\nabla_2,\Fil_2)$ be two filtered de Rham bundles over $(X,D)/S$. Then \[\mHom((V_1,\nabla_1,\Fil_1),(V_2,\nabla_2,\Fil_2)) := \left(\mHom(V_1,V_2),\nabla_{(\nabla_1,\nabla_2)},\Fil_{(\Fil_1,\Fil_2)}\right)\]
is also a filtered de Rham bundle.
\end{lem}
\begin{proof}
	We only need to check that the connection and the filtration satisfy Griffith transversality. This follows ii) of Lemma~\ref{main_lemma}.
\end{proof}
\begin{rmk}
	Let $(V,\nabla,\Fil)$ be a filtered de Rham bundle. By abuse of notation, we often write $\nabla_{(\nabla,\nabla)}$ as $\nabla^\End$ and $\Fil_{(\Fil,\Fil)}$ as $\Fil$ for short.
\end{rmk}

\begin{lem}\label{diff_conn} Let $(V,\Fil)$ be a filtered vector bundle over $X$. Suppose $\nabla$ and $\nabla'$ are two connections on $V$ over $(X,D)/S$ with $\nabla-\nabla'\in \Fil^{-1}_{(\Fil,\Fil)} \Hom(V,V\otimes \Omega^1_{X/S}(\log D))$. Then $\nabla$ satisfies Griffith's transversality if and only if $\nabla'$ does. In other words, $(V,\nabla,\Fil)$ is a filtered de Rham bundle if and only if $(V,\nabla',\Fil)$ is. 
\end{lem}
\begin{proof} We only need to show that if $\nabla$ satisfies Griffith's transversality then $\nabla'$ also does. Let $v$ be any local section in $\Fil^\ell V$. By the Griffith's transversality, $\nabla(v)\in \Fil^{\ell-1}V\otimes \Omega^1_{X/S}(\log D)$. By Lemma~\ref{main_lemma}, $\nabla-\nabla'\in \Fil^{-1}_{(\Fil,\Fil)} \Hom(V,V\otimes \Omega^1_{X/S}(\log D))$ implies $\nabla(v)-\nabla'(v) \in \Fil^{\ell-1}V\otimes \Omega^1_{X/S}(\log D)$. Thus 
	\[\nabla'(v) = \nabla(v)-\Big(\nabla(v)-\nabla'(v)\Big) \in \Fil^{\ell-1}V\otimes \Omega^1_{X/S}(\log D).\] 
	By the arbitrary choice of the local section $v$, $\nabla'$ satisfies Griffith's transversality.
\end{proof}

\subsection{Graded (logarithmic) Higgs bundles}

\paragraph{\emph{Graded vector bundles.}}
Let $X$ be a scheme. Let $E$ be a vector bundle on $X$. Let $\{\mathrm{Gr}^\ell E\}_{\ell\in \mathbb Z}$ be subbundles of $V$. The pair $(E,\mathrm{Gr})$ is called \emph{graded vector bundle} over $X$ if the natural map $ \oplus_{\ell\in \mathbb Z} \mathrm{Gr}^\ell E\cong E$ is an isomorphism. 
\begin{lem}
For any graded vector bundle $(E,\mathrm{Gr})$ on $X$ there exists an open affine covering $\{U_i\}_{i\in I}$ of $X$ such that $\mathrm{Gr}^\ell\mid_{U_i}$ is a finite free $\mathcal O_{U_i}$-module for each $i\in I$ and $\ell\in \mathbb Z$.	
\end{lem} 
\begin{proof}
	There are only finitely many non-zero $\mathrm{Gr}^{\ell}$, so we may choose an open affine covering that trivializes them all.
\end{proof}

Let $(E,\Gr)$ be a graded vector bundle over a scheme $X$. There is a natural grading structure $\mathrm{Gr}^{\vee}$ on the dual vector bundle $E^\vee$ given by 
\begin{equation}\label{dual_Gr}
\mathrm{Gr}^{\vee \ell}(E^\vee) = (\mathrm{Gr}^{-\ell} E)^\vee.
\end{equation}
For two graded vector bundles $(E_1,\mathrm{Gr}_1)$ and $(E_2,\mathrm{Gr}_2)$ on $X$, there is a natural grading structure $\Gr_{(\Gr_1,\Gr_2)}$ on the Hom vector bundle $\mHom(E_1,E_2)=E_1^\vee\otimes E_2$ given by 
\begin{equation}\displaystyle \label{Hom_Gr}
\mathrm{Gr}^\ell_{(\mathrm{Gr}_1,\mathrm{Gr}_2)} \mHom(E_1,E_2) = \bigoplus_{\ell_1\in \mathbb{Z}} (\mathrm{Gr}^{\ell_1}V_1)^\vee \otimes_R \mathrm{Gr}^{\ell + \ell_1}V_2.
\end{equation}

\paragraph{\emph{Higgs bundles.}}
Let $X$ be a smooth $S$-scheme with relative flat normal crossing divisor $D$. Let $E$ be vector bundle over $X$. Let $\theta\colon E\rightarrow E\otimes_{\mathcal O_X} \Omega^1_{X/S}(\log D)$ be an $\mathcal O_X$-linear morphism. The pair $(E,\theta)$ is called \emph{(logarithmic) Higgs bundle} over $(X,D)/S$ if $\theta$ is integrable. i.e. $\theta\wedge\theta = 0$. For a Higgs bundle, one has a natural \emph{Higgs complex}
\[\DR(E,\theta):\qquad 0\rightarrow  E 
\xrightarrow{\theta} E\otimes\Omega^1_{X/S}(\log D)
\xrightarrow{\theta} E\otimes\Omega^2_{X/S}(\log D)
\xrightarrow{\theta} E\otimes\Omega^3_{X/S}(\log D) \rightarrow \cdots \]
The hypercohomology group of $\DR(E,\theta)$ is called \emph{Higgs cohomology} of the Higgs bundle $(E,\theta)$, which is denoted by
\[H^i_{Higgs}(X,(E,\theta)) := \mathbb H^i(X,\DR(E,\theta)).\]

\paragraph{\emph{Graded Higgs bundles.}}
Let $X$ be an $S$-scheme with relative flat normal crossing divisor $D$. A \emph{graded (logarithmic) Higgs bundle} over $(X,D)/S$ is a Higgs bundle $(E,\theta)$ together with a grading structure $\mathrm{Gr}$ on $E$ satisfying
\[\theta(\mathrm{Gr}^\ell E) \subset \mathrm{Gr}^{\ell -1}E\otimes_{\mathcal O_X} \Omega^1_{X/S}(\log D).\]
Thus we have subcomplexes of $\DR(E,\theta)$
\[\Gr^\ell \DR(E,\theta): 
\qquad 0
\rightarrow           \Gr^\ell E
\xrightarrow{\theta}  \Gr^{\ell-1} E \otimes \Omega^1_{X/S}(\log D)
\xrightarrow{\theta}  \Gr^{\ell-2} E \otimes \Omega^2_{X/S}(\log D)
\xrightarrow{\theta}  \cdots \]
These subcomplexes define a graded structure on the complex $\DR(E,\theta)$
\begin{equation}
\DR(E,\theta) = \bigoplus_{\ell\in \mathbb Z} \Gr^{\ell} \DR(E,\theta).
\end{equation}

The following is the main example we will be concerned with.
\begin{exa} \label{exa1}
Let $(V,\nabla,\Fil)$ be a filtered de Rham bundle over $V$. Denote $E =\bigoplus_{\ell\in \mathbb Z} \Fil^\ell V/\Fil^{\ell+1}V$ and $\Gr^\ell E = \Fil^\ell V/\Fil^{\ell+1}V$. By Griffith's transversality, the connection induces an $\mathcal O_X$-linear map $\theta \colon \Gr^\ell E \rightarrow \Gr^{\ell-1} E  \otimes_{\mathcal O_X} \Omega^1_{X/S}(\log D)$ for each $\ell\in \mathbb Z$. Then $(E,\theta,\Gr)$ is a graded Higgs bundle.	Moreover we have 
\[\Gr^{\ell} \DR(E,\theta) = \Fil^\ell \DR(V,\nabla) / \Fil^{\ell+1} \DR(V,\nabla).\]
\end{exa} 

\subsection{Cech resolution}

\paragraph{\emph{Double complex.}} Recall that a double complex is an array $M^{pq}$ together with morphisms 
\begin{equation*}
\left\{\begin{split}
& \mathrm{d} \colon   M^{pq} \rightarrow M^{p  ,q+1}\\
& \mathrm{d'}\colon  M^{pq} \rightarrow M^{p+1,q  }\\
\end{split} \right.
\end{equation*}
satisfying $\mathrm{d}\circ\mathrm{d'} + \mathrm{d'}\circ \mathrm{d} =0$.
\begin{equation}
\xymatrix{
& \vdots \ar[d]^{\mathrm{d}'} 
& \vdots \ar[d]^{\mathrm{d}'}
& 
\\
\cdots       \ar[r]^{\mathrm{d}}
& M^{p,q}    \ar[r]^{\mathrm{d}}  \ar[d]^{\mathrm{d}'}  
& M^{p,q+1}  \ar[r]^{\mathrm{d}}  \ar[d]^{\mathrm{d}'} 
&\cdots \\
\cdots       \ar[r]^{\mathrm{d}}
& M^{p+1,q}  \ar[r]^{\mathrm{d}}  \ar[d]^{\mathrm{d}'} 
& M^{p+1,q+1}\ar[r]^{\mathrm{d}}  \ar[d]^{\mathrm{d}'} 
& \cdots \\
& \vdots & \vdots & \\
}
\end{equation}
The \emph{total complex} is 
\begin{equation}
\cdots 
\xrightarrow{\mathrm{d}^{\mathrm{Tot}}} \bigoplus_{p+q = n  } M^{p,q} \xrightarrow{\mathrm{d}^{\mathrm{Tot}}} \bigoplus_{p+q = n+1} M^{p,q}  \xrightarrow{\mathrm{d}^{\mathrm{Tot}}} \bigoplus_{p+q = n+2} M^{p,q} 
\xrightarrow{\mathrm{d}^{\mathrm{Tot}}} \cdots
\end{equation}
where $\mathrm{d}^{\mathrm{Tot}} = \mathrm{d} +\mathrm{d}'$. i.e. $\mathrm{d}^{\mathrm{Tot}}((m^{p,q})_{p+q=n}) := (\mathrm{d}(m^{p,q-1}) + \mathrm{d}'(m^{p-1,q}))_{p+q=n+1}$.

One example of a double complex comes from the \emph{Cech complex} attached to a complex of sheaves and an open covering. Let $X$ be a topological space. Let 
\[(\mathcal F^\bullet,\nabla): \qquad  0
\rightarrow \mathcal F^0 
\xrightarrow{\nabla} \mathcal F^1
\xrightarrow{\nabla} \mathcal F^2 
\xrightarrow{\nabla} \mathcal F^3 
\xrightarrow{\nabla} \mathcal F^4 
\xrightarrow{\nabla} \cdots \]
be a complex of sheaves of abelian groups on $X$. Let $\mathcal U = \{U_i\}_{i\in I}$ be an open covering of $X$. For any $\iota=(i_1,\cdots,i_s)\in I^s$, denote $U_{\iota} = \cap_{\ell =1}^s U_{i_\ell}$ and 
\[C^p(\mathcal U,\mathcal F^q) = \prod_{\iota\in I^{p+1}} \mathcal F^q(U_{\iota}).\]

Then the \emph{Cech resolution} of the complex $(\mathcal F^\bullet,\nabla)$ with respect to the open covering $\mathcal U$ of $X$ is the total complex of the following double complex
 \begin{equation}
 \xymatrix{
 	C^0(\mathcal U,\mathcal F^0)
 	\ar[d]^{\delta} \ar[r]^{\nabla} 
 	&  C^0(\mathcal U,\mathcal F^1)
 	\ar[d]^{\delta} \ar[r]^{\nabla} 
 	&  C^0(\mathcal U,\mathcal F^2)
 	\ar[d]^{\delta} \ar[r]^{\nabla} 
 	&  C^0(\mathcal U,\mathcal F^3)
 	\ar[d]^{\delta} \ar[r]^{\nabla} 
 	&  \cdots  \\
 	C^1(\mathcal U,\mathcal F^0)
 	\ar[d]^{\delta} \ar[r]^{-\nabla} 
 	&  C^1(\mathcal U,\mathcal F^1)
 	\ar[d]^{\delta} \ar[r]^{-\nabla} 
 	&  C^1(\mathcal U,\mathcal F^2)
 	\ar[d]^{\delta} \ar[r]^{-\nabla} 
 	&  C^1(\mathcal U,\mathcal F^3)
 	\ar[d]^{\delta} \ar[r]^{-\nabla} 
 	&  \cdots  \\
 	   C^2(\mathcal U,\mathcal F^0) 
 	\ar[d]^{\delta} \ar[r]^{\nabla} 
 	&  C^2(\mathcal U,\mathcal F^1)
 	\ar[d]^{\delta} \ar[r]^{\nabla} 
 	&  C^2(\mathcal U,\mathcal F^2)
 	\ar[d]^{\delta} \ar[r]^{\nabla} 
 	&  C^2(\mathcal U,\mathcal F^3)
 	\ar[d]^{\delta} \ar[r]^{\nabla} 
 	&  \cdots  \\
 	C^3(\mathcal U,\mathcal F^0)
 	\ar[d]^{\delta} \ar[r]^{-\nabla} 
 	&  C^3(\mathcal U,\mathcal F^1)
 	\ar[d]^{\delta} \ar[r]^{-\nabla} 
 	&  C^3(\mathcal U,\mathcal F^2)
 	\ar[d]^{\delta} \ar[r]^{-\nabla} 
 	&  C^3(\mathcal U,\mathcal F^3)
 	\ar[d]^{\delta} \ar[r]^{-\nabla} 
 	&  \cdots  \\
 	\vdots &  \vdots & \vdots & \vdots &  \\
 }
 \end{equation}
 where $\delta$ is defined by restrictions.
 \[\delta((s_{\iota})_{\iota\in I^m}) = \left(\sum_{t=0}^{p+1}(-1)^t s_{(i_0,\cdots,\hat{i}_t,\cdots,i_{p+1})}\mid_{U_{(i_0,i_1,\cdots,i_{p+1})}}\right)_{(i_0,i_1,\cdots,i_{p+1})\in I^{m+1}}\]
 \begin{lem}\label{cocycle & coboundary}
 \begin{itemize}
 	\item[(1).]  An element $(-a,b,c)$ in $C^2(\mathcal U,\mathcal F^0)\times C^1(\mathcal U,\mathcal F^1) \times C^0(\mathcal U,\mathcal F^2)$ is a $2$-cocycle if and only if 
 	\[0 = \delta(a), \nabla(a)=\delta(b), \nabla(b)=\delta(c) \text{ and } \nabla(c)=0;\]
 	\item[(2).] an element $(a,b)$ in $C^1(\mathcal U,\mathcal F^0)\times C^0(\mathcal U,\mathcal F^1)$ $1$-cocycle if and only if 
 	\[0 = \delta(a), \nabla(a)=\delta(b) \text{ and } \nabla(b)=0.\]
 \end{itemize}   
 \end{lem}

\subsection{Complexes associated to a filtered de Rham bundle.} \label{complexes_(V,nabla,Fil)}
Let $S$, $\w{S}$, $\mathfrak a$, $\w{X}$, $\w D$, $X$, $D$, $\mathcal I$ be as in Setup~\ref{setup}. Let $\w V$ be a vector bundle on $\w X$. We denote $V = \w V\otimes_{\mathcal O_{\w X}}\mathcal O_X$ and  $V_{\mathcal I} = V\otimes_{\mathcal O_X} \mathcal I$.
%

\paragraph{\emph{Canonical isomorphism.}}

\begin{lem}\label{CanoIsom} Let $\w V$ be a vector bundle on $\w X$. Then 
	\begin{itemize}
		\item[(1).] there is an canonical isomorphism of sheaves over the underlying space of $X$ and $\w X$ 
		\[V_{\mathcal I} \xrightarrow{\sim} \mathcal I\cdot \w V \subset \w V\]
		which maps $v \otimes r$ to $r \w v$,  for any $v\in V$, $r\in \mathcal I$ and any lift $\w v$ of $v\in V$ in $\w V$.
		\item[(2).] if $\w v\in \w V$ is a local section, then $\w v$ is a section of the image of the sheaf $V_{\mathcal I}$ if and only if $\w v \equiv 0 \pmod{\mathcal I}$.
	\end{itemize}
\end{lem}
\begin{rmk}
	The isomorphism in (1) of Lemma~\ref{CanoIsom} is well defined, i.e., it does not depend on the choice of $\w v$.
\end{rmk}

\paragraph{\emph{Natural connection on $\mathcal I$.}}
Consider the restriction of $\mathrm{d}\colon \mathcal O_{\w X}\rightarrow \Omega^1_{\w X/\w S}(\log \w D)$ to $\mathcal I$ 
\[\mathrm{d} \colon \mathcal I\rightarrow \mathcal I\cdot \Omega^1_{\w X/\w S}(\log \w D) \simeq \mathcal I \otimes \Omega^1_{X/S}(\log D).\] One checks that the pair $(\mathcal I,\mathrm{d})$ is a de Rham sheaf over $(X,D)/S$.  Note that it is not a de Rham bundle as $\mathcal I$ is not a vector bundle on $X$. 

\paragraph{\emph{de Rham bundle twisted by $(\mathcal I,\mathrm{d})$}.} Let $(V,\nabla)$ be a de Rham bundle over $(X,D)/S$.  Taking tensor product, one gets a tensor de Rham sheaf \[(V,\nabla)_{\mathcal I} := (V,\nabla)\otimes (\mathcal I,\mathrm{d})\]
over $(X,D)/S$ with connection $\nabla_{\mathcal I} \colon V_{\mathcal I} \rightarrow V_{\mathcal I} \otimes \Omega^1_{X/S}(\log D)$ defined by $\nabla_{\mathcal I}(v\otimes r):= \nabla(v) \otimes r + v\otimes \mathrm{d} r$. 

\paragraph{\emph{The complex $\DR(\mEnd(V,\nabla)_{\mathcal I})$.}}
Let $(V,\nabla)$ be a de Rham bundle over $(X,D)/S$.  By Lemma~\ref{Hom_(V,nabla)}, $\mEnd(V,\nabla)$ is also a de Rham bundle over $(X,D)/S$. Consider 
\[\mEnd(V,\nabla)_{\mathcal I}:=\mEnd(V,\nabla)\otimes (\mathcal I,\mathrm{d}),\]
which is a de Rham sheaf over $(X,D)/S$. Then its de Rham complex $\DR(\mEnd(V,\nabla)_\mathcal I)$ is 
\begin{equation} \label{complex_(V,nabla)} 
0 \rightarrow \mEnd(V)_{\mathcal I} \overset{\oln^\End}{\longrightarrow} 
\mEnd(V)_{\mathcal I} \otimes\Omega_{X/S}^1(\log D)  \overset{\oln^\End}{\longrightarrow} 
\mEnd(V)_{\mathcal I} \otimes\Omega_{X/S}^2(\log D)  \overset{\oln^\End}{\longrightarrow} \cdots
\end{equation}
where $\oln^\End = \left(\nabla_{(\nabla,\nabla)}\right)_{\mathcal I}$ in terms of our earlier notation in Equation \ref{eqn:hom_connection}.
\begin{lem} \label{CanoInj}
 Suppose there are three de Rham bundles $(\w V_1,\w\nabla_1)$, $(\w V_2,\w\nabla_2)$ and $(\w V_3,\w\nabla_3)$ over $(\w X,\w D)/\w S$ which lift $(V,\nabla)$, and there are morphisms $f_{12}\colon \w V_1\rightarrow \w V_2$ and $f_{23}\colon \w V_2\rightarrow \w V_3$ lifting the identity map on $V$. Then 
 \begin{itemize}
 	\item[(1).] there are  short exact sequence of de Rham sheaves over $\w X$
 	\[0\rightarrow \mEnd(V,\nabla)_{\mathcal I} \xrightarrow{\iota_{(\w V_i,\w V_j)}} \mHom((\w V_i,\w\nabla_i),(\w V_j,\w\nabla_j)) \longrightarrow \mEnd(V,\nabla)\rightarrow 0.\]
 	In particular, $\iota^{-1}_{(\w V_i,\w V_j)} \circ \nabla_{(\w\nabla_i,\w\nabla_j)} = \oln^\End \circ \iota^{-1}_{(\w V_i,\w V_j)}$ on $\mathcal I\cdot\mHom(\w V_i,\w V_i) = \iota_{(\w V_i,\w V_j)}(\mEnd(V,\nabla)_{\mathcal I})$.
 	\item[(2).]  One has $\iota_{(\w V_1,\w V_3)} \circ \iota_{(\w V_2,\w V_3)}^{-1} = (-\circ f_{12})$ on $\mathcal I\cdot\mHom(\w V_1,\w V_2)$ and $\iota_{(\w V_1,\w V_3)} \circ \iota_{(\w V_1,\w V_2)}^{-1} = (f_{23}\circ -)$  on $\mathcal I\cdot\mHom(\w V_1,\w V_2)$. i.e. the following diagram communicates
 	\begin{equation*}
 	\xymatrix@C=3cm{
 		\mEnd(V)_\mathcal I \ar@{^(->}[r]^-{\iota_{(\w V_1,\w V_2)}} \ar@{=}[d] 
 		&
 		 \mHom(\w V_1,\w V_2)  \ar[d]^{f_{23}\circ -}\\
 		\mEnd(V)_\mathcal I \ar@{^(->}[r]^-{\iota_{(\w V_1,\w V_3)}} 
 		&
 		 \mHom(\w V_1,\w V_3) \\
 		\mEnd(V)_\mathcal I \ar@{^(->}[r]^-{\iota_{(\w V_2,\w V_3)}} \ar@{=}[u] 
 		&
 		 \mHom(\w V_2,\w V_3) \ar[u]_{-\circ f_{12}}\\
 	}
 	\end{equation*}
 	
 \end{itemize}

\end{lem}
\begin{proof}
Lemma~{CanoIsom} implies (1) and the map $\iota_{(\w V_i,\w V_j)}$ satisfies
\[ \iota_{(\w V_i,\w V_j)}(g\otimes r) = r\cdot\w g\]
where $g$ is a local section of $\mEnd(V)$, $r$ is a local section of $\mathcal I$ and $\w g$ is a local section of $\mHom(\w V_i,\w V_j)$ which lifts $g$.

Let $\w g_{23}$ be a lifting of $g$ in $\mHom(\w V_2,\w V_3)$. Then $\w g_{23}\circ f_{12}$ is an element in $\mHom(\w V_2,\w V_3)$ which also lifts $f$. By the construction of $\iota_{(\w V_i,\w V_j)}$, one has 
\[\iota_{(\w V_1,\w V_3)}(g\otimes r) = r\cdot (\w g_{23}\circ f_{12}) = \iota_{(\w V_2,\w V_3)}(g\otimes r) \circ f_{12}.\]
Similarly one has 
\[\iota_{(\w V_1,\w V_3)}(g\otimes r) = r\cdot (f_{23}\circ \w g_{12}) = f_{23}\circ \iota_{(\w V_1,\w V_2)}(g\otimes r).\] 
Thus (2) follows.
\end{proof}
\paragraph{\emph{Filtered de Rham bundle twisted by $(\mathcal I,\mathrm{d},\Fil_{\mathrm{tri}})$}.} Suppose $(V,\nabla,\Fil)$ is a filtered de Rham bundle. Taking the tensor product with $(\mathcal I,\mathrm{d},\Fil_{\mathrm{tri}})$, one obtains a filtered de Rham sheaf on $(X,D)/S$. 
\[(V,\nabla,\Fil)_{\mathcal I} := (V,\nabla,\Fil)\otimes (\mathcal I,\mathrm{d},\Fil_{\mathrm{tri}})\]
where the underlying de Rham bundle is $(V,\nabla)_\mathcal I$ and the filtration $\Fil_\mathcal I$ is defined by $\Fil^\ell_\mathcal I(V_\mathcal I) := (\Fil^\ell V)_\mathcal I$. For simplifying the notations, we will write $\Fil^\ell_\mathcal I(V_\mathcal I)$ as $\Fil^\ell V_\mathcal I$. 


\paragraph{\emph{The filtration on the complex $\DR(\mEnd(V,\nabla)_{\mathcal I})$.}} For a filtered de Rham bundle $(V,\nabla,\Fil)$ over $(X,D)/S$, The filtration on
\[\mEnd(V,\nabla,\Fil)_{\mathcal I} := \mEnd(V,\nabla,\Fil)\otimes (\mathcal I,\mathrm{d},\Fil_{\mathrm{tri}}).\]
induces a filtration on $\DR(\mEnd(V,\nabla)_{\mathcal I})$
\[  \cdots
\subset   \Fil^{\ell+1}\DR(\mEnd(V,\nabla)_{\mathcal I}) 
\subset   \Fil^{\ell  }\DR(\mEnd(V,\nabla)_{\mathcal I}) 
\subset   \Fil^{\ell-1}\DR(\mEnd(V,\nabla)_{\mathcal I}) 
\subset \cdots \subset \DR(\mEnd(V,\nabla)_{\mathcal I}) \]
where $\Fil^{\ell}\DR(\mEnd(V,\nabla)_{\mathcal I})$ is defined as
\begin{equation}\label{complex_(V,nabla,Fil)}
 0  \rightarrow 
 \Fil^{\ell}\mEnd(V)_{\mathcal I} 
 \overset{\oln^\End}{\longrightarrow} 
 \Fil^{\ell-1}\mEnd(V)_{\mathcal I}
 \otimes\Omega_{X/S}^1(\log D)
 \overset{\oln^\End}{\longrightarrow} 
 \Fil^{\ell-2}\mEnd(V)_{\mathcal I}\otimes\Omega_{X/S}^2(\log D) 
\overset{\oln^\End}{\longrightarrow} \cdots
 \end{equation}

\begin{lem}\label{CanoInj_Fil}
	 Suppose $(\w V_1,\w\nabla_1,\w\Fil_1)$ and $(\w V_2,\w\nabla_2,\w\Fil_2)$ are de Rham bundles over $(\w X,\w D)/\w S$ which lift $(V,\nabla,\Fil)$. For each $\ell\in \mathbb Z$, one has an exact sequence over $\w X$
	\[0\rightarrow \Fil^\ell\mEnd(V)_{\mathcal I} \xrightarrow{\iota_{(\w V_1,\w V_2)}} \Fil^\ell_{(\w\Fil_1,\w\Fil_2)}\mHom(\w V_1,\w V_2) \longrightarrow \Fil^\ell\mEnd(V)\rightarrow 0.\]
\end{lem}
\begin{proof}This follows Lemma~\ref{CanoIsom}.
\end{proof}
\paragraph{\emph{The complex $\mathscr C$.}} We define a complex $\mathscr C$ which will govern the deformation theory. Let $(V,\nabla,\Fil)$ be a filtered de Rham bundle over $(X,D)/S$. Denote by $\mathscr C$ the cokernel of the natural injection $\Fil^0\DR(\mEnd(V,\nabla)_\mathcal I)\hookrightarrow \DR(\mEnd(V,\nabla)_\mathcal I)$. i.e. $\mathscr C:=\frac{\DR(\mEnd(V,\nabla)_\mathcal I)}{\Fil^0 \DR(\mEnd(V,\nabla)_\mathcal I)}$
\begin{equation}  \label{complex_(Fil)}
0 \rightarrow 
\frac{\DR(\mEnd(V,\nabla)_\mathcal I)}{\Fil^0\DR(\mEnd(V,\nabla)_\mathcal I)} \overset{\nabla^\End}{\longrightarrow} 
\frac{\DR(\mEnd(V,\nabla)_\mathcal I)}{\Fil^{-1}\DR(\mEnd(V,\nabla)_\mathcal I)} \otimes\Omega_{X/S}^1(\log D)
\overset{\nabla^\End}{\longrightarrow}
\frac{\DR(\mEnd(V,\nabla)_\mathcal I)}{\Fil^{-2}\DR(\mEnd(V,\nabla)_\mathcal I)} \otimes\Omega_{X/S}^2(\log D)
\overset{\nabla^\End}{\longrightarrow} \cdots
\end{equation}

\subsection{Complexes associated to a graded Higgs bundle.} 
Let $S$, $\w{S}$, $\mathfrak a$, $\w{X}$, $\w D$, $X$, $D$, $\mathcal I$ be as in Setup~\ref{setup}. One has a natural Higgs sheaf $(\mathcal I,0)$, i.e.,  the sheaf $\mathcal I$ together with the trivial Higgs field. 

\paragraph{\emph{The complex $\DR(\mEnd(E,\theta)_{\mathcal I})$.}}
Let $(E,\theta)$ be a Higgs bundle over $(X,D)/S$. Then $\mEnd(E,\theta):=(\mEnd(E),\theta^\End)$ is also Higgs bundle over $(X,D)/S$, where
$\theta^\End(f) = \theta\circ f - f\circ \theta$. Consider the Higgs sheaf
\[\mEnd(E,\theta)_{\mathcal I}:= \mEnd(E,\theta)\otimes (\mathcal I,0),\]
 over $(X,D)/S$, where $\overline{\theta}^\End(f\otimes r)= \theta^\End(f) \otimes r$. Then its Higgs complex $\DR(\mEnd(E,\theta)_\mathcal I)$ is 
\begin{equation} \label{complex_(E,nabla)} 
0 \rightarrow 
\overset{\overline{\theta}^\End}{\longrightarrow} 
\mEnd(E)_{\mathcal I} \otimes\Omega_{X/S}^1(\log D)  
\overset{\overline{\theta}^\End}{\longrightarrow} 
\mEnd(E)_{\mathcal I} \otimes\Omega_{X/S}^2(\log D)  
\overset{\overline{\theta}^\End}{\longrightarrow} \cdots
\end{equation}

\paragraph{\emph{The grading structure on the complex $\DR(\mEnd(E,\theta)_{\mathcal I})$.}}
Suppose $(E,\theta,\Gr)$ is a graded Higgs bundle. Then $\Gr$ induces a natural graded structure on $\mEnd(E,\theta)_{\mathcal I}$ given by
\[\Gr^\ell(\mEnd(E)_{\mathcal I}):=(\Gr^\ell\mEnd(E))_{\mathcal I} \]
This grading structure induces a grading structure on the complex $\DR(\mEnd(V,\nabla)_{\mathcal I})$
\[\DR(\mEnd(E,\theta)_{\mathcal I}) = \bigoplus_{l\in \mathbb{Z}}  \Gr^{\ell}\DR(\mEnd(E,\theta)_{\mathcal I}),\] 
where $\Gr^{\ell}\DR(\mEnd(E,\theta)_{\mathcal I})$ is defined as
\begin{equation}\label{complex_(E,theta,Gr)}
0  \rightarrow 
\Gr^{\ell}\mEnd(E)_{\mathcal I} 
\overset{\overline{\theta}^\End}{\longrightarrow} 
\Gr^{\ell-1}\mEnd(E)_{\mathcal I}
\otimes\Omega_{X/S}^1(\log D)
\overset{\overline{\theta}^\End}{\longrightarrow} 
\Gr^{\ell-2}\mEnd(E)_{\mathcal I}\otimes\Omega_{X/S}^2(\log D) 
\overset{\overline{\theta}^\End}{\longrightarrow} \cdots
\end{equation}

\begin{lem}\label{exa2}
Let $(V,\nabla,\Fil)$ be a filtered de Rham bundle over $V$. Let $(E,\theta,\Gr)$ be the graded Higgs bundle defined in Example~\ref{exa1}. Then 
\[\Gr^{\ell}\DR(\mEnd(E,\theta)_{\mathcal I}) = \Fil^{\ell}\DR(\mEnd(V,\nabla)_{\mathcal I})/\Fil^{\ell+1}\DR(\mEnd(V,\nabla)_{\mathcal I}).\]
In particular, the complex $\mathscr C$ in~\ref{complex_(Fil)} is an successive extension of $\Gr^{-1}\DR(\mEnd(E,\theta)_{\mathcal I}),\Gr^{-2}\DR(\mEnd(E,\theta)_{\mathcal I}),\cdots$.
\end{lem}

\subsection{Local Liftings}
In this subsection, we show that a filtered vector bundle with a Griffith's tranverse filtration locally lifts. We will later use this fact to construct a tangent-obstruction theory for filtered de Rham bundles. Let $S$, $\w{S}$, $\mathfrak a$, $\w{X}$, $\w D$, $X$, $D$, $\mathcal I$ be given as in Setup~\ref{setup}.

\begin{lem} \label{local_lifting_(V,Fil)}
Let $(V,\nabla,\Fil)$ be a filtered vector bundle over $X$ with a connection satisfying Griffith's transversality. Suppose that $\Fil^\ell V$ is free over $\mathcal O_X$ for all $\ell\in \mathbb Z$. Then
\begin{itemize}
	\item[(1).] there exists lifting $\w V$ over $\w X$ of $V$;
	\item[(2).] for any given lifting  $\w V$ over $\w X$ of $V$, there exists a filtration $\w\Fil$ on $\w V$ which lifts $\Fil$;
	\item[(3).] for any given lifting $(\w V,\w\Fil)$ over $\w X$ of the filtered vector bundle $(V,\Fil)$, there exists a (not necessarily integrable) connection $\w\nabla$ on $\w V$ over $(\w X,\w D)/\w S$ which lifts $\nabla$ and satisfies Griffith's transversality with respect to $\w \Fil$.
	\item[(4).] Let $(\w V,\w\Fil)$ and $(\w V',\w\Fil')$ be two liftings of $(V,\Fil)$ over $\w X$. Then there exists an isomorphism $f\colon (\w V,\w\Fil)\rightarrow (\w V',\w\Fil')$ which lifts the identity map $\mathrm{id}_{V}$.
\end{itemize} 
\end{lem}
\begin{proof} Part (1) follows the freeness of $V$. For the part (2), by the freeness of $\Fil^\ell V$, there exists an adapted basis  $\{e_1,\cdots,e_r\}$ of $(V,\Fil)$. i.e. 
	\[V = \bigoplus_{j=1}^r \mathcal O_X \cdot e_i \quad \text{and} \quad \Fil^\ell V = \bigoplus_{j: e_j \in \Fil^\ell V} \mathcal O_X\cdot e_j.\]
    Fix a lifting $\w e_j$ of $e_j$ in $\w V$ for any $j=1,\cdots,r$.  Then the filtration on $\w V$ defined by
	\[\w\Fil^\ell \w V = \bigoplus_{j: e_j \in \Fil^\ell V} \mathcal O_{\w X}\cdot \w e_j\]
	is a lifting of $\Fil$ on $\w V$.
	
	For part (3), choose an adopted basis $\{\w e_1,\cdots,\w e_r\}$ of $\w V$ and consider the connection $\w{\mathrm{d}}$ on $\w V$ defined by $\w{\mathrm{d}}(\w e_j)=0$. It obviously satisfies Griffith's transversality. Denote $\mathrm{d} = \w{\mathrm{d}}\mid _V$. Since $\nabla$ also satisfies Griffith's transversality and $\nabla -\mathrm{d}$ is $\mathcal O_X$-linear, $\nabla - \mathrm{d} \in \Fil^{-1} \Hom(V,V\otimes \Omega^1_{X/S}(\log D))$ by Lemma~\ref{main_lemma}. By Lemma~\ref{CanoInj_Fil}, the map
	\[\Fil^{-1} \Hom(\w V,\w V\otimes \Omega^1_{\w X/\w S}(\log \w D)) \twoheadrightarrow  \Fil^{-1} \Hom(V,V\otimes \Omega^1_{X/S}(\log D))\]
	is surjective. Hence there exists an element $\w\omega$ in $\Fil^{-1} \Hom(\w V,\w V\otimes \Omega^1_{\w X/\w S}(\log \w D))$ which lifts $\nabla - \mathrm{d}$.
	 Consider the connection $\w\nabla:= \w{\mathrm{d}} + \w\omega$ on $\w V$, which is a lifting of $\nabla$. By Lemma~\ref{diff_conn}, $\w\nabla$ satisfies Griffith's transversality. 
	
	We construct an $f$ in the part (4) as following. Since $\Fil^\ell V$ is free, $\w \Fil^\ell \w V$ and $\w \Fil'^\ell \w V'$ are also free with the same rank. We may lift $\{e_1,\cdots,e_r\}$ to an adapted basis $\{\w e_1,\cdots,\w e_r\}$ of $(\w V,\w\Fil)$ and to an adopted basis $\{\w e_1',\cdots,\w e_r'\}$ of $(\w V',\w\Fil')$. The isomorphism $f\colon \w V \rightarrow \w V'$ sending $\w e_i$ to $\w e_i'$ preserves the filtrations, which is what we need.
\end{proof}

\section{Deformations of filtered de Rham bundles and Hodge filtrations}
In this section we will study the deformation theory of filtered de Rham bundles and Hodge filtrations. \footnote{Note that every de Rham bundle admits a natural filtration: the \emph{trivial filtration}. Therefore the deformation theory of filtered de Rham bundles generalizes the deformation theory of de Rham bundles.} The main results are Theorem~\ref{deformation_(V,nabla,Fil)} and Theorem~\ref{deformation_(Fil)}.

Let $S$, $\w{S}$, $\mathfrak a$, $\w{X}$, $\w D$, $X$, $D$, $\mathcal I$ be given as in Setup~\ref{setup}.

\subsection{Deforming a filtered de Rham bundle}

In this subsection, we study the deformation theory of a filtered de Rham bundle.  For a filtered de Rham bundle $(V,\nabla,\Fil)$ over $(X,D)/S$, recall the complex of sheaves $\Fil^0 \DR(\mEnd(V,\nabla)_\mathcal I)$ defined in~\ref{complex_(V,nabla,Fil)}:
\begin{equation*} 
0 
\rightarrow 
\Fil^0 \mEnd(V)_{\mathcal I} 
\overset{\oln^\End}{\longrightarrow} 
\Fil^{-1} \mEnd(V)_{\mathcal I} \otimes\Omega_{X/S}^1(\log D)  
\overset{\oln^\End}{\longrightarrow} 
\Fil^{-2} \mEnd(V)_{\mathcal I} \otimes\Omega_{X/S}^1(\log D)  
\overset{\oln^\End}{\longrightarrow} \cdots
\end{equation*}
This complex yields a tangent-obstruction theory for a filtered de Rham bundle.
\begin{thm}\label{deformation_(V,nabla,Fil)} Let $(V,\nabla,\Fil)$ be a (logarithmic) filtered de Rham bundle over $(X,D)/S$. Then
	\begin{itemize}
		\item[(1).] the obstruction to lifting $(V,\nabla,\Fil)$ to a filtered de Rham bundle over $(\w X,\w D)/\w S$ lies in $\mathbb H^2\big(\Fil^0 \DR(\mEnd(V,\nabla)_\mathcal I)\big)$;
		\item[(2).] if $(V,\nabla,\Fil)$ has lifting $(\widetilde{V},\widetilde{\nabla},\w\Fil)$, then the lifting set is an $\mathbb H^1\big(\Fil^0 \DR(\mEnd(V,\nabla)_\mathcal I)\big)$-torsor;
		\item[(3).] the infinitesimal automorphism group of $(\widetilde{V},\widetilde{\nabla},\w\Fil)$ over $(V,\nabla,\Fil)$ is $\mathbb H^0\big(\Fil^0 \DR(\mEnd(V,\nabla)_\mathcal I)\big)$.
	\end{itemize}
\end{thm} 

In particular, when $\Fil$ is the trivial filtration, Theorem~\ref{deformation_(V,nabla,Fil)} implies the classical result about deforming integrable connections.
\begin{prop}\label{deformation_(V,nabla)} Let $(V,\nabla)$ be a (logarithmic) de Rham bundle over $(X,D)/S$. Then
	\begin{itemize}
		\item[1).] the obstruction to lifting $(V,\nabla)$ to a de Rham bundle over $(\w X,\w D)/\w S$ lies in $\mathbb H^2\big(\DR(\mEnd(V,\nabla)_\mathcal I)\big)$;
		\item[2).] if $(V,\nabla)$ has lifting $(\widetilde{V},\widetilde{\nabla})$, then the lifting set is an $\mathbb H^1\big(\DR(\mEnd(V,\nabla)_\mathcal I)\big)$-torsor;
		\item[3).] the infinitesimal automorphism group of $(\widetilde{V},\widetilde{\nabla})$ over $(V,\nabla)$ is $\mathbb H^0\big(\DR(\mEnd(V,\nabla)_\mathcal I)\big)$.
	\end{itemize}
\end{prop} 
The following easily follows from the earlier sections and is the key we need to define an obstruction class.
\begin{lem} \label{local_data_(V,nabla,Fil)} Let $(V,\nabla,\Fil)$ be a filtered vector bundle together with an connection over $(X,D)/S$ satisfying Griffith transversality. 
	\begin{itemize}
		\item [(1).] there exists an open affine covering $\{U_i\}_{i\in I}$ of $X$ such that $\Fil^\ell V(U_i)$ is free over $\mathcal O_X(U_i)$ for each $i\in I$ and each $\ell\in \mathbb Z$;
		\item [(2).] the restriction $(V,\Fil)\mid_{U_i}$ extends to a filtered vector bundle $(\w V_i,\w\Fil_i)$ over $\w U_i$ with $\w\Fil_i^\ell\w V_i(\w U_i)$ free over $\mathcal O_{\w X}(\w U_i)$ for each $i\in I$ and each  $\ell\in \mathbb Z$;
		\item [(3).] for any two $i,j\in I$, there exists a local isomorphism $f_{ij}\colon (\w V_i,\w\Fil_i)\mid_{\w U_{ij}}\rightarrow (\w V_j,\w\Fil_j)\mid_{\w U_{ij}}$ lifts the identity $\mathrm{id}_{V\mid_{U_{ij}}}$;
		\item [(4).] the $\nabla\mid_{U_i}$ extends to a connection $\w{\nabla}_i$ on $\w V_i$ which satisfies Griffith transversality with respect to the filtration $\w\Fil_i$.
	\end{itemize}
	We call a quadruple $\big(\{U_i\}_i,\{\w V_i,\w\Fil_i\}_i,\{f_{ij}\}_{ij},\{\w{\nabla}_i\}_i\big)$ a \emph{system of local data} for $(V,\nabla,\Fil)$. 
\end{lem}

\begin{proof}[Proof of Lemma~\ref{local_data_(V,nabla,Fil)}]
The lemma follows Lemma~\ref{local freeness_(V,Fil)} and Lemma~\ref{local_lifting_(V,Fil)}.
\end{proof}
\begin{rmk}
	In (4) of Lemma~\ref{local_data_(V,nabla,Fil)}, we do not require the local extension $\w\nabla\mid_{\w U_i}$ to be integrable.
\end{rmk}
Let $(V,\nabla,\Fil)$ be a filtered de Rham bundle and $\big(\{U_i\}_i,\{\w V_i,\w\Fil_i\}_i,\{f_{ij}\}_{ij},\{\w{\nabla}_i\}_i\big)$ be the local data outputted by Lemma~\ref{local_data_(V,nabla,Fil)}. Since $f_{ij}$ preserves local filtration and $\w\nabla_i$ satisfies Griffith's transversality, by Lemma~\ref{main_lemma} the element 
\begin{equation}\label{obs:c(v,nabla,Fil)}
\w c(V,\nabla,\Fil):= \Big(-(f_{jk}\circ f_{ij}-f_{ik})_{(i,j,k)}, (\w \nabla_j\circ f_{ij} - f_{ij}\circ\w\nabla_i)_{(i,j)}, (\w\nabla_i\circ\w\nabla_i)_{i}\Big)
\end{equation}
is contained in 
\[
\prod_{(i,j,k)\in I^3} 
\left(\w\Fil^0_{ik}\mHom(\w V_i,\w V_k)\right)(\w U_{ijk})
\times 
\prod_{(i,j)\in I^2}
\left(\w\Fil_{ij}^{-1}\mHom(\w V_i,\w V_j)\otimes\Omega^1_{\w X/\w S}(\log \w D)\right)(\w U_{ij})
\times 
\prod_{i\in I} \left(\w\Fil_{ii}^{-2}\mEnd(\w V_i)\otimes\Omega^2_{\w X/\w S}(\log \w D)\right)(\w U_{i}),
\]
where $\w\Fil_{ij} = \Fil_{(\w\Fil_i,\w\Fil_j)}$ is the filtration on the vector bundle $\mHom(\w V_i,\w V_j)$ over $\w U_{ij}$.
Since  $(V,\nabla,\Fil)$ is a globally defined filtered de Rham bundle, $c(V,\nabla,\Fil)=0 \mod \mathcal I$. One gets an element
\[c(V,\nabla,\Fil) := \Big(-\big(\iota^{-1}_{ik}(f_{jk}\circ f_{ij}-f_{ik})\big)_{(i,j,k)}, \big(\iota^{-1}_{ij}((\w \nabla_j\circ f_{ij} - f_{ij}\circ\w\nabla_i)\big)_{(i,j)}, \big(\iota^{-1}_{ii}(\w\nabla_i\circ\w\nabla_i)\big)_{i}\Big)\]
in
\[
\prod_{(i,j,k)\in I^3}
\Fil^0\mEnd(V)_{\mathcal I}({U_{ijk}}) 
\times 
\prod_{(i,j)\in I^2}   
\left(\Fil^{-1}\mEnd(V)_{\mathcal I} \otimes\Omega^1_{X/S}(\log D)\right)({U_{ij}})
\times 
\prod_{i\in I}
\left(\Fil^{-2}\mEnd(V)_{\mathcal I} \otimes\Omega^2_{X/S}(\log D)\right)({U_{i}}).\]
where the notations $\iota_{ij} :=\iota_{\w V_i,\w V_j}$ are defined in Lemma~\ref{CanoInj}.  
\begin{lem}\label{2-cocycle_(V,nabla,Fil)} Let $(V,\nabla,\Fil)$ be a filtered de Rham bundle and $\big(\{U_i\}_i,\{(\w V_i,\w\Fil_i)\}_i,\{f_{ij}\}_{ij},\{\w{\nabla}_i\}_i\big)$ be a system of local data, whose existence is guaranteed by Lemma~\ref{local_data_(V,nabla,Fil)}.
	\begin{itemize}
		\item [(1).] The element $c(V,\nabla,\Fil)$ defined in (\ref{obs:c(v,nabla,Fil)}) is a $2$-cocycle in the Cech resolution of the  complex $\Fil^0 \DR(\mEnd(V,\nabla)_\mathcal I)$ with respect to the open covering $\mathcal U = \{U_i\}_{i\in I}$ of $X$.
		\item [(2).] The class $[c(V,\nabla,\Fil)] \in \mathbb H^2(\Fil^0 \DR(\mEnd(V,\nabla)_\mathcal I))$ does not depend on the choice of the system of local data from Lemma~\ref{local_data_(V,nabla,Fil)}.
	\end{itemize}
	
\end{lem}
\begin{proof}
We denote $\mathcal F^k=\Fil^{-k}\mEnd(V)_{\mathcal I} \otimes \Omega^k_{X/S}(\log D)$ for all $k\geq 0$. Then the complex 
$\Fil^0 \DR(\mEnd(V,\nabla)_\mathcal I)$ defined in \ref{complex_(V,nabla,Fil)} may be rewritten as
\[0\rightarrow \mathcal F^0 \xrightarrow{\oln^\End}\mathcal F^1 \xrightarrow{\oln^\End}\mathcal F^2\xrightarrow{\oln^\End} \mathcal F^3 \xrightarrow{\oln^\End} \cdots  \]
For any $\iota=(i_1,\cdots,i_s)\in I^s$, set $U_{\iota} = \cap_{\ell =1}^s U_{i_\ell}$ and 
\[C^p(\mathcal U,\mathcal F^k) = \prod_{\iota\in I^{p+1}} \mathcal F^k(U_{\iota}).\]
Recall that the Cech resolution of the complex $\Fil^0 \DR(\mEnd(V,\nabla)_\mathcal I)$ with respect to the open covering $\mathcal U$ of $X$ is the total complex of the following double complex
\begin{equation}
\xymatrix{
   C^0(\mathcal U,\mathcal F^0)
   \ar[d]^{\delta} \ar[r]^-{\oln^\End} 
&  C^0(\mathcal U,\mathcal F^1)
   \ar[d]^{\delta} \ar[r]^-{\oln^\End} 
&  \red{ C^0(\mathcal U,\mathcal F^2) }
   \ar[d]^{\delta} \ar[r]^-{\oln^\End} 
&  C^0(\mathcal U,\mathcal F^3)
   \ar[d]^{\delta} \ar[r]^-{\oln^\End} 
&  \cdots  \\
   C^1(\mathcal U,\mathcal F^0)
\ar[d]^{\delta} \ar[r]^-{-\oln^\End} 
& \red{ C^1(\mathcal U,\mathcal F^1) }
\ar[d]^{\delta} \ar[r]^-{-\oln^\End} 
&  C^1(\mathcal U,\mathcal F^2)
\ar[d]^{\delta} \ar[r]^-{-\oln^\End} 
&  C^1(\mathcal U,\mathcal F^3)
\ar[d]^{\delta} \ar[r]^-{-\oln^\End} 
&  \cdots  \\
   \red{ C^2(\mathcal U,\mathcal F^0) }
\ar[d]^{\delta} \ar[r]^-{\oln^\End} 
&  C^2(\mathcal U,\mathcal F^1)
\ar[d]^{\delta} \ar[r]^-{\oln^\End} 
&  C^2(\mathcal U,\mathcal F^2)
\ar[d]^{\delta} \ar[r]^-{\oln^\End} 
&  C^2(\mathcal U,\mathcal F^3)
\ar[d]^{\delta} \ar[r]^-{\oln^\End} 
&  \cdots  \\
   C^3(\mathcal U,\mathcal F^0)
\ar[d]^{\delta} \ar[r]^-{-\oln^\End} 
&  C^3(\mathcal U,\mathcal F^1)
\ar[d]^{\delta} \ar[r]^-{-\oln^\End} 
&  C^3(\mathcal U,\mathcal F^2)
\ar[d]^{\delta} \ar[r]^-{-\oln^\End} 
&  C^3(\mathcal U,\mathcal F^3)
\ar[d]^{\delta} \ar[r]^-{-\oln^\End} 
&  \cdots  \\
   \vdots &  \vdots & \vdots & \vdots &  \\
}
\end{equation}
By lemma~\ref{cocycle & coboundary}, to show that
\[c(V,\nabla,\Fil)= \Big(
-\big(\iota^{-1}_{ik}(f_{jk}\circ f_{ij}-f_{ik})\big)_{(i,j,k)},
 \big(\iota^{-1}_{ij}((\w \nabla_j\circ f_{ij} - f_{ij}\circ\w\nabla_i)\big)_{(i,j)},
 \big(\iota^{-1}_{ii}(\w\nabla_i\circ\w\nabla_i)\big)_{i}\Big)\] 
forms a cocycle, one needs to check
\begin{equation*}
\left\{
\begin{array}{lclll}
0 & = & \delta\left( \big(\iota^{-1}_{ik}(f_{jk}\circ f_{ij}-f_{ik})\big)_{(i,j,k)} \right) & \in  C^3(\mathcal U,\mathcal F^0)\\
\oln^\End \left( \big(\iota^{-1}_{ik}(f_{jk}\circ f_{ij}-f_{ik})\big)_{(i,j,k)} \right) & = & \delta\left(\big(\iota^{-1}_{ij}((\w \nabla_j\circ f_{ij} - f_{ij}\circ\w\nabla_i)\big)_{(i,j)}\right)  & \in  C^2(\mathcal U,\mathcal F^1)\\
\oln^\End\left(\big(\iota^{-1}_{ij}((\w \nabla_j\circ f_{ij} - f_{ij}\circ\w\nabla_i)\big)_{(i,j)}\right) & = &\delta\left(\big(\iota^{-1}_{ii}(\w\nabla_i\circ\w\nabla_i)\big)_{i}\Big)\right)  & \in  C^1(\mathcal U,\mathcal F^2) \\
\oln^\End\left(\big(\iota^{-1}_{ii}(\w\nabla_i\circ\w\nabla_i)\big)_{i}\Big)\right)  &= & 0 & \in   C^0(\mathcal U,\mathcal F^3)\\
\end{array}
\right.
\end{equation*}
This is equivalent to check
\begin{itemize}
	\item[1a).] $\iota_{jl}^{-1} (f_{kl}\circ f_{jk}-f_{jl}) 
	- \iota_{il}^{-1} (f_{kl}\circ f_{ik}-f_{il}) 
	+ \iota_{il}^{-1} (f_{jl}\circ f_{ij}-f_{il}) 
	- \iota_{jk}^{-1} (f_{jk}\circ f_{ij}-f_{ik}) =0$ for each $(i,j,k,l)\in I^4$.
	\item[1b).] 
	$\iota_{ik}^{-1}(\w\nabla_{ik}(f_{jk}\circ f_{ij}-f_{ik})) = 
	\iota_{jk}^{-1}(\w\nabla_k\circ f_{jk}- f_{jk} \circ \w\nabla_j) - 
	\iota_{ik}^{-1}(\w\nabla_k\circ f_{ik} - f_{ik}\circ \w\nabla_i) +
	\iota_{ij}^{-1}(\w\nabla_j\circ f_{ij} - f_{ij}\circ \w\nabla_i)$ 
	for each $(i,j,k)\in I^3$.
	\item[1c).] $\iota_{ij}^{-1}(\w\nabla_{ij}(\w \nabla_j\circ f_{ij} - f_{ij}\circ\w\nabla_i))
	=
	\iota_{jj}^{-1}(\w\nabla_j\circ\w\nabla_j) -
	\iota_{ii}^{-1}(\w\nabla_i\circ\w\nabla_i)$ for each $(i,j)\in I^2$; and
	\item[1d).] $\w\nabla_{ii}(\w\nabla_i\circ\w\nabla_i) = 0$.
\end{itemize}
where $\w\nabla_{ij}:=\nabla_{(\w\nabla_i,\w\nabla_j)}$ be the connection on $\mHom(\w V_i,\w V_j)\mid_{\w U_{ij}}$.

1a) follows Lemma~\ref{CanoInj} and 
\begin{equation*}
(f_{kl}\circ f_{jk}-f_{jl})\circ f_{ij} 
- (f_{kl}\circ f_{ik}-f_{il}) 
+ (f_{jl}\circ f_{ij}-f_{il}) 
-  f_{kl} \circ (f_{jk}\circ f_{ij}-f_{ik})
= 0
\end{equation*}
This is because that Lemma 2.10 implies $-\circ f_{ij} = \iota_{ik} \circ \iota^{-1}_{ij}$ and $f_{jk}\circ - = \iota_{ik} \circ \iota^{-1}_{jk}$. 1a) follows the injectivity of $\iota_{il}$ and 
\begin{equation*}
\begin{split}
0
&=(f_{kl}\circ f_{jk}-f_{jl})\circ f_{ij} 
- (f_{kl}\circ f_{ik}-f_{il}) 
+ (f_{jl}\circ f_{ij}-f_{il}) 
-  f_{kl} \circ (f_{jk}\circ f_{ij}-f_{ik}) \\
&\overset {\ref{CanoInj}}{=} 
  \iota_{il} \circ \iota^{-1}_{jl} (f_{kl}\circ f_{jk}-f_{jl})
- \iota_{il} \circ \iota^{-1}_{il} (f_{kl}\circ f_{ik}-f_{il}) 
+ \iota_{il} \circ \iota^{-1}_{il} (f_{jl}\circ f_{ij}-f_{il}) 
- \iota_{il} \circ \iota^{-1}_{ik} (f_{jk}\circ f_{ij}-f_{ik})\\
& = \iota_{il}\circ\left( 
  \iota^{-1}_{jl} (f_{kl}\circ f_{jk}-f_{jl})
- \iota^{-1}_{il} (f_{kl}\circ f_{ik}-f_{il}) 
+ \iota^{-1}_{il} (f_{jl}\circ f_{ij}-f_{il}) 
- \iota^{-1}_{ik} (f_{jk}\circ f_{ij}-f_{ik})\right)\\
\end{split}
\end{equation*}
1b) follows Lemma~\ref{CanoInj} and 	
\begin{equation*}
\begin{split}	
    \w\nabla_{ik}(f_{jk}\circ f_{ij}-f_{ik}) 
& := \w\nabla_k\circ(f_{jk}\circ f_{ij}-f_{ik}) 
    - (f_{jk}\circ f_{ij}-f_{ik}) \circ \w\nabla_i\\
& = (\w\nabla_k\circ f_{jk}- f_{jk} \circ \w\nabla_j) \circ f_{ij}  
- (\w\nabla_k\circ f_{ik} - f_{ik}\circ \w\nabla_i)
+   f_{jk}\circ (\w\nabla_j\circ f_{ij} - f_{ij}\circ \w\nabla_i)\\ 
\end{split}
\end{equation*}
 
1c) follows Lemma~\ref{CanoInj} and 
\begin{equation*}
\begin{split}	
\w\nabla_{ij}(\w \nabla_j\circ f_{ij} - f_{ij}\circ\w\nabla_i) 
& := \w \nabla_j \circ (\w \nabla_j\circ f_{ij} - f_{ij}\circ\w\nabla_i) - (\w \nabla_j\circ f_{ij} - f_{ij}\circ\w\nabla_i)  \circ  \w \nabla_i \\
& = \w \nabla_j^2 \circ f_{ij} - f_{ij}\circ \w \nabla_i^2 \\ 
\end{split}
\end{equation*}
1d) follow that  $\w\nabla_{ii}(\w\nabla_i\circ\w\nabla_i)=\w\nabla_i\circ(\w\nabla_i\circ\w\nabla_i) - (\w\nabla_i\circ\w\nabla_i)\circ\w\nabla_i = 0$.

Let $\big(\{U_i\}_{i\in J},\{(\w V_i,\w\Fil_i)\}_{i\in J},\{f_{ij}\}_{i,j\in J},\{\w{\nabla}_i\}_{i\in J}\big)$ be another system of local data. Consider the union covering $\{U_i\}_{i\in K}$, where $K=I\sqcup J$.  Using the natural embedding $\sigma_1\colon I \rightarrow K$ and $\sigma_2 \colon J\rightarrow K$, the coverings $\{U_i\}_{i\in I}$ and $\{ U_i\}_i\in J$ are each refinements of $\{U_{i}\}_{i\in K}$. Thus one has following diagram
\begin{equation}
\xymatrix{ 
	H^2(\mathcal U_I, (\mathcal F^\cdot,\nabla^{\End})) 
	&H^2(\mathcal U_K, (\mathcal F^\cdot,\nabla^{\End}))
	\ar[r]^{\sigma_2^*}\ar[l]_{\sigma_1^*}
	&H^2(\mathcal U_J, (\mathcal F^\cdot,\nabla^{\End}))}
\end{equation}
As in Lemma~\ref{local_data_(V,nabla,Fil)}, we extend the two systems of local data $\big(\{U_i\}_{i\in I},\{(\w V_i,\w\Fil_i)\}_{i\in I},\{f_{ij}\}_{i,j\in I},\{\w{\nabla}_i\}_{i\in I}\big)$ and $\big(\{U_i\}_{i\in J},\{(\w V_i,\w\Fil_i)\}_{i\in J},\{f_{ij}\}_{i,j\in J},\{\w{\nabla}_i\}_{i\in J}\big)$ to a system of local data $\big(\{U_i\}_{i\in K},\{(\w V_i,\w\Fil_i)\}_{i\in K},\{f_{ij}\}_{i,j\in K},\{\w{\nabla}_i\}_{i\in K}\big)$. Then
$\sigma_1^*(c_K(V,\nabla,\Fil)) = c_I(V,\nabla,\Fil)$ and
$\sigma_2^*(c_K(V,\nabla,\Fil)) = c_J(V,\nabla,\Fil)$. Thus $ c_I(V,\nabla,\Fil)$ and $ c_J(V,\nabla,\Fil)$ defines the same element in $\mathbb H^2(\mathcal F^\cdot,\oln^{\End}) := \varinjlim\limits_{\mathcal U} H^2(\mathcal U,(\mathcal F^\cdot,\oln^{\End}))$.
\end{proof}

\begin{lem}\label{local_isomorphism_(V,nabla,Fil)} Let $(\w V,\w \nabla,\w\Fil)$ and $(\w V,\w \nabla,\w\Fil)'$ be two liftings of a filtered de Rham bundle $(V,\nabla,\Fil)$. Then
	\begin{itemize}
		\item [(1).]  there exists an open affine covering  $\{U_i\}_{i\in I}$ such that $\w\Fil^\ell\w V(\w U_i)$ and $\w\Fil^\ell\w V'(\w U_i)$ are free over $\mathcal O_{\w X}(\w U_i)$ for all $i\in I$ and $\ell\in \mathbb Z$, and 
		\item [(2).] there exists a local isomorphism $f_i\colon (\w V,\w\Fil)\mid_{\w U_i} \rightarrow (\w V,\w\Fil)'\mid_{\w U_i}$, which lifts the identity $\mathrm{id}_{V\mid_{U_i}}$.
	\end{itemize}
\end{lem}

\begin{proof}
	The Lemma follows Lemma~\ref{local freeness_(V,Fil)} and Lemma~\ref{local_lifting_(V,Fil)}.
\end{proof}

Suppose there are two liftings $(\w V,\w \nabla,\w\Fil)$ and $(\w V,\w \nabla,\w\Fil)'$ of a filtered de Rham bundle $(V,\nabla,\Fil)$ and $\big(\{U_i\}_i,\{f_{i}\}_{i}\big)$ be a system of local data given as in Lemma~\ref{local_isomorphism_(V,nabla,Fil)}. Denote $\iota=\iota_{\w V,\w V'}$, which is defined in Lemma~\ref{CanoInj}. By similar reason as $c(V,\nabla,\Fil)$, one constructs an element 
\begin{equation}\label{torsor_(V,nabla,Fil)}
b((\w V,\w\nabla,\w\Fil),(\w V,\w\nabla,\w\Fil)'):=\Big(\big(\iota^{-1}(f_j-f_i)\big)_{(i,j)},\big(\iota^{-1}(\w\nabla'\circ f_{i}-f_i\circ\w\nabla)\big)_{i}\Big)
\end{equation}
in 
\[
\prod_{(i,j)\in I^2} \Fil^0\mEnd(V)_{\mathcal I}\mid_{U_{ij}} \times 
\prod_{i\in I}   \left(\Fil^{-1}\mEnd(V)_{\mathcal I} \otimes\Omega^1_{X/S}(\log D)\right) \mid_{U_{i}}.\]

\begin{lem}\label{1-cocycle_(V,nabla,Fil)} Let $(\w V,\w \nabla,\w\Fil)$ and $(\w V,\w \nabla,\w\Fil)'$ be two liftings of a filtered de Rham bundle $(V,\nabla,\Fil)$. Then
	\begin{itemize}
		\item [(1).] the element $b((\w V,\w\nabla,\w\Fil),(\w V,\w\nabla,\w\Fil)')$ defined in (\ref{torsor_(V,nabla,Fil)}) is a $1$-cocycle in the Cech resolution of the  complex $\Fil^0\DR(\mEnd(V,\nabla))_{\mathcal I}$ with respect to the open covering $\{U_i\}_{i\in I}$ of $X$. 
		\item [(2).] The class $[b((\w V,\w\nabla,\w\Fil),(\w V,\w\nabla,\w\Fil)')] \in \mathbb H^1(\Fil^0\DR(\mEnd(V,\nabla))_{\mathcal I})$ does not depend on the choice of  $\big(\{U_i\}_i,\{f_{i}\}_{i}\big)$.
	\end{itemize}
\end{lem}
\begin{proof}
To show
$b((\w V,\w\nabla,\w\Fil),(\w V,\w\nabla,\w\Fil)'):=
\Big(
\big(\iota^{-1}(f_j-f_i)\big)_{(i,j)},
\big(\iota^{-1}(\w\nabla'\circ f_{i}-f_i\circ\w\nabla)\big)_{i}
\Big)$ 
is a $1$-cocycle, we only need to check
\begin{equation*}
\left\{
\begin{array}{lclll}
0 & = & \delta\left( \big(\iota^{-1}(f_j-f_i)\big)_{(i,j)} \right) & \in  C^2(\mathcal U,\mathcal F^0)\\
\oln^\End\left( \big(\iota^{-1}(f_j-f_i)\big)_{(i,j)} \right)  & = & \delta \left( \big(\iota^{-1}(\w\nabla'\circ f_{i}-f_i\circ\w\nabla)\big)_{i} \right) & \in  C^1(\mathcal U,\mathcal F^1) \\
\oln^\End\left( \big(\iota^{-1}(\w\nabla'\circ f_{i}-f_i\circ\w\nabla)\big)_{i} \right) & = & 0 & \in   C^0(\mathcal U,\mathcal F^2)\\
\end{array}
\right.
\end{equation*}
This is equivalent to check
\begin{itemize}
	\item[1a).] $(f_k-f_j)\mid_{U_{ijk}} - (f_k-f_i)\mid_{U_{ijk}} + (f_j-f_i)\mid_{U_{ijk}} = 0$ for each $(i,j,k)\in I^3$.
	\item[1b).] 
	$\nabla_{(\w\nabla,\w\nabla')}(f_j-f_i) 
	= (\w\nabla'\circ f_{j}-f_j\circ\w\nabla)\mid_{U_{ij}} - (\w\nabla'\circ f_{i}-f_i\circ\w\nabla)\mid_{U_{ij}}$ 
	for each $(i,j)\in I^2$.
	\item[1c).]  $\nabla_{(\w\nabla,\w\nabla')} (\w\nabla'\circ f_{i}-f_i\circ\w\nabla)= 0$ for each $i\in I$.
\end{itemize}
 1a) and 1b) are trivial. 1c) follows from the integrability of $\w\nabla'$ and $\w\nabla$.
 
Follow the same method as  in the proof of (2) of Lemma~\ref{2-cocycle_(V,nabla,Fil)}, one gets the proof of (2). 
\end{proof}

\begin{lem}\label{0-cocycle_(V,nabla,Fil)} Let $(\w V,\w \nabla,\w\Fil)$ be a lifting of a filtered de Rham bundle $(V,\nabla,\Fil)$. Denote $\iota = \iota_{(\w V,\w V)}$. 
\begin{itemize}		
	\item [(1).] For any $\epsilon \in \mathbb H^0\big(\Fil^0\DR(\mEnd(V,\nabla)_\mathcal I)\big)$, $\mathrm{id}_{\w V} + \iota(\epsilon)$ is an automorphism of $(\w V,\w \nabla,\w\Fil)$ which lifts $\mathrm{id}_{V}$.
	\item [(2).] Let $f$ be an automorphism of $(\w V,\w \nabla,\w\Fil)$ which lifts $\mathrm{id}_{V}$. Then
	\[\iota^{-1}(f-\mathrm{id}_{\w V}) \in \mathbb H^0\big(\Fil^0\DR(\mEnd(V,\nabla)_\mathcal I)\big).\]
\end{itemize}
\end{lem}
\begin{proof} Note that
\begin{equation*}
\begin{split}
\mathbb H^0\big(\DR(\mEnd(V,\nabla)_\mathcal I)\big) 
& = 
\left\{ \epsilon \in \Gamma\left(X,\Fil^0\mEnd(V)_{\mathcal I}\right) 
\mid
\oln^{\End}(\epsilon)=0.\right\} \\ 
& = 
\left\{ \iota(f-\mathrm{id}) \mid f \in \Gamma(X,\Fil^0_{(\w\Fil,\w\Fil)}\mEnd(\w V)) \text{ with } f\equiv \mathrm{id}\pmod{\mathcal I} \text{ and } \nabla_{(\w\nabla,\w\nabla)}(f)=0.\right\}\\
\end{split}
\end{equation*} 
The Lemma follows that fact that for any $ f \colon \w V\rightarrow \w V$ with $f \equiv \mathrm{id}_{\w V}\pmod{\mathcal I}$
\begin{itemize}
	\item $f$ preserves the filtration if and only if $\eta \in \Fil^0_{(\w\Fil,\w\Fil)}\mEnd(\w V)$ (Lemma~\ref{main_lemma});
	\item $f$ is parallel if and only if $\nabla_{(\w\nabla,\w\nabla)}(\eta) =0$. \qedhere
\end{itemize}
\end{proof}

\begin{proof}[Proof of Proposition~\ref{deformation_(V,nabla,Fil)}] According Lemma~\ref{2-cocycle_(V,nabla,Fil)}, Lemma~\ref{1-cocycle_(V,nabla,Fil)} and Lemma~\ref{0-cocycle_(V,nabla,Fil)} one only need to show that 
\begin{itemize}
	\item[(a)] if the filtered de Rham  bundle $(V,\nabla,\Fil)$ has a lifting $(\w V,\w\nabla,\w\Fil)$, then $[c(V,\nabla,\Fil)]=0$;
	\item[(b)] if $[c(V,\nabla,\Fil)]=0$, then the filtered de Rham  bundle $(V,\nabla,\Fil)$ is liftable;
	\item[(c)] Let $(\w V,\w\nabla,\w\Fil)$ be a lifting of the filtered de Rham bundle $(V,\nabla,\Fil)$. For any element $\epsilon$ in $\mathbb H^1(\Fil^0\DR(\mEnd(V,\nabla)_{\mathcal I}))$, there exists another lifting $(\w V,\w\nabla,\w\Fil)'$ such that $\epsilon = [b((\w V,\w\nabla,\w\Fil),(\w V,\w\nabla,\w\Fil)')]$.
\end{itemize}
The proof of (a) is trivial: if we choose a system of local data $\big(\{U_i\}_i,\{(\w V_i,\w\Fil_i)\}_i,\{f_{ij}\}_{ij},\{\w{\nabla}_i\}_i\big)$ such that $(\w V_i,\w\nabla_i,\w\Fil_i) = (\w V,\w\nabla,\w\Fil)\mid_{\w U_i}$ and $f_{ij} = \mathrm{id}_{\w V\mid_{\w U_{ij}}}$, then $c(V,\nabla,\Fil) =0$. 
\paragraph{Proof of (b).} Since 
$c(V,\nabla,\Fil)=\Big(
-\big(\iota^{-1}_{ik}(f_{jk}\circ f_{ij}-f_{ik})\big)_{(i,j,k)}, 
\big(\iota^{-1}_{ij}((\w \nabla_j\circ f_{ij} - f_{ij}\circ\w\nabla_i)\big)_{(i,j)},
\big(\iota^{-1}_{ii}(\w\nabla_i\circ\w\nabla_i)\big)_{i}
\Big)$
 is a coboundary, there exists an element 
 $\Big(
 \big(\iota^{-1}_{ij}(\Delta f_{ij})\big)_{(i,j)},
 \big(\iota^{-1}_{ii}(\Delta \omega_i)\big)_{i}
 \Big)$
in 
\[
\prod_{(i,j)\in I^2} \left(\Fil^0\mEnd(V)\right)_{\mathcal I}\mid_{U_{ij}} \times 
\prod_{i\in I}   \left(\left(\Fil^{-1}\mEnd(V)\right)_{\mathcal I} \otimes\Omega^1_{X/S}(\log D)\right) \mid_{U_{i}}.\]	
such that 
\begin{equation*}
\left\{
\begin{array}{rcl}
 -\big(\iota^{-1}_{ik}(f_{jk}\circ f_{ij}-f_{ik})\big)_{(i,j,k)}  
 & = & \delta \left( \big(\iota^{-1}_{ij}(\Delta f_{ij})\big)_{(i,j)} \right)\\
 \big(\iota^{-1}_{ij}((\w \nabla_j\circ f_{ij} - f_{ij}\circ\w\nabla_i)\big)_{(i,j)}
 & = & - \oln^\End \left( \big(\iota^{-1}_{ij}(\Delta f_{ij})\big)_{(i,j)} \right) +\delta \left( \big(\iota^{-1}_{ii}(\Delta \omega_i)\big)_{i} \right)\\
\big(\iota^{-1}_{ii}(\w\nabla_i\circ\w\nabla_i)\big)_{i}
 & = & \oln^\End \left( \big(\iota^{-1}_{ii}(\Delta \omega_i)\big)_{i} \right)\\
\end{array}
\right.
\end{equation*}
This is equivalent to
\begin{equation*}
\left\{
\begin{array}{lll}
f_{jk}\circ f_{ij}-f_{ik} + (\Delta f_{jk}\circ f_{ij} - \Delta f_{ik} + f_{jk}\circ\Delta f_{ij}) =0 & \text{ for each } & (i,j,k)\in I^3\\
\w \nabla_j\circ f_{ij} - f_{ij}\circ\w\nabla_i + (\w\nabla_j\circ \Delta f_{ij} -\Delta f_{ij} \circ\w\nabla_i) - (\Delta \omega_j \circ f_{ij} - f_{ij}\circ \Delta \omega_i) = 0 & \text{ for each } & (i,j)\in I^2\\
\w\nabla_i\circ\w\nabla_i - (\w\nabla_i\circ \Delta \omega_i + \Delta\omega_i\circ\w\nabla_i)=0 & \text{ for each } &  i\in I\\
\end{array}
\right.
\end{equation*}
Denote $\w\nabla'_i = \w\nabla_i - \Delta \omega_i$ and $f'_{ij} = f_{ij} + \Delta f_{ij}$.
Then it is equivalent to
\begin{equation*}
\left\{
\begin{array}{lll}
f'_{jk}\circ f'_{ij}=f'_{ik} & \text{ for each } & (i,j,k)\in I^3\\
\w \nabla'_j\circ f'_{ij} =f'_{ij}\circ\w\nabla'_i & \text{ for each } & (i,j)\in I^2\\
\w\nabla'_i\circ\w\nabla'_i = 0 & \text{ for each } &  i\in I\\
\end{array}
\right.
\end{equation*}
This implies that local data $(\w V_i,\w\nabla_i)_i$ can be glued via $(f_{ij}')_{ij}$ into a de Rham bundle $(\w V,\w\nabla)$ over $(\w X,\w D)/\w S$. 
 
Since $f_{ij}\in\w\Fil_{ij}^0\mHom(\w V_i,\w V_j)$ and 
\[\Delta f_{ij}\in \Fil^0\mEnd(V)_{\mathcal I}(U_{ij})\subset \w\Fil_{ij}^0\mHom(\w V_i,\w V_j),\]
$f'_{ij}$ is also contained in $\w\Fil_{ij}^0\mHom(\w V_i,\w V_j)$. By Lemma~\ref{main_lemma} $f_{ij}'$'s preserve local filtrations $\w\Fil_i$'s. Thus local filtrations $(\w\Fil_i)_i$ can be glued into a filtration $\w\Fil$ on $\w V$. Since $(\w\nabla_i,\w \Fil_i)$ satisfying Griffith transversality and 
\[\Delta \omega_{i} \in \Big(\Fil^{-1}\mEnd(V)_{\mathcal I}\otimes \Omega^1_{X/S}(\log D)\Big)(U_i)\subset \w\Fil_{ii}^{-1}\mEnd(\w V_i)\otimes \Omega^1_{\w X/\w S}(\log \w D)(\w U_i),\]
$(\w\nabla'_i,\w \Fil_i)$ is also satisfying Griffith transversality by Lemma~\ref{diff_conn}.
In conclusion, we get a filtered de Rham bundle $(\w V,\w\nabla,\w\Fil)$ which lifts $(V,\nabla,\Fil)$.

\paragraph{Proof of (c).} Denote $\iota=\iota_{(\w V,\w V)}$. Let $\Big(
\big(\iota^{-1}(\Delta f_{ij})\big)_{(i,j)},
\big(\iota^{-1}(\Delta \omega_i)\big)_{i}
\Big)$ be a representation of the $1$-class $\epsilon$ with respect to a covering $\{\w U_i\}_{i\in I}$ of $\w X$,
where
\[\Delta f_{ij} \in \mathcal I \cdot \Fil^0_{(\w\Fil,\w\Fil)}\mEnd(\w V) \text{ and } \Delta \omega_i\in \mathcal I \cdot \Fil^{-1}_{(\w\Fil,\w\Fil)}\mEnd(\w V)\otimes \Omega^1_{\w X/\w S}(\log \w D).\]
Denote $(\w V_i,\w\nabla_i,\w\Fil_i):= (\w V,\w\nabla,\w\Fil)\mid_{\w U_i}$ and $f_{ij}=\mathrm{id}_{\w V\mid_{\w U_{ij}}}$. Then 
\begin{equation*}
\left\{
\begin{array}{rclll}
f_{jk}\circ f_{ij}-f_{ik} 
& = 0 = & 
\Delta f_{jk}\circ f_{ij} - \Delta f_{ik} + f_{jk}\circ\Delta f_{ij}, 
& \text{ for each } 
& (i,j,k)\in I^3; \\
-(\w \nabla_j\circ f_{ij} - f_{ij}\circ\w\nabla_i) 
& = 0 = &
- (\w\nabla_j\circ \Delta f_{ij} -\Delta f_{ij} \circ\w\nabla_i) + (\Delta \omega_j \circ f_{ij} - f_{ij}\circ \Delta \omega_i)
& \text{ for each } 
& (i,j)\in I^2; \\
-\w\nabla_i\circ\w\nabla_i 
& = 0 = &  
\w\nabla_i\circ \Delta \omega_i + \Delta\omega_i\circ \w\nabla_i,
& \text{ for each } 
& i\in I. \\
\end{array}
\right.
\end{equation*}

Similar as in the proof of (b), denote $\w\nabla_i'= \w\nabla + \Delta\omega_i$ and $f_{ij}' = \mathrm{id}-\Delta f_{ij}$ one glues $(\w V_i,\w\nabla_i',\w\Fil_i)$ via $f_{ij}'$ into a filtered de Rham bundle $(\w V,\w \nabla,\w \Fil)'$ on $\w X$. The similar method in the proof of (2) in Lemma~\ref{2-cocycle_(V,nabla,Fil)} show that up to isomorphism this filtered de Rham bundle does not depend on the choice of the representation of $\epsilon$. In the following we show that 
\[b((\w V,\w\nabla,\w\Fil),(\w V,\w\nabla,\w\Fil)') =\Big(\big(\iota^{-1}(\Delta f_{ij})\big)_{(i,j)},\big(\iota^{-1}(\Delta \omega_i)\big)_{i}\Big).\]
Since $(\w V',\w\nabla')$ is glued from $(\w V\mid_{\w U_i},\w\nabla'_i)$ via local isomorphism $f'_{ij}$, there exists local isomoprhism $f_i\colon \w V_{\w U_i}\rightarrow \w V'_{\w U_i}$ which lifts $\mathrm{id}_{V\mid_{U_i}}$ such that $f_i = f_j\circ f'_{ij}$ and $\w\nabla'\circ f_i = f_i\circ \w\nabla'_i$.
\begin{equation*}
\xymatrix{
	\w V\mid_{\w U_{ij}} \ar[r]^{f_i} \ar@/_10pt/[d]_{\mathrm{id}} \ar@{..>}[dr] \ar@/^10pt/[d]^{f'_{ij}} \ar@{..>}[dr] & 
	\w V'\mid_{\w U_{ij}} \ar[d]^{\mathrm{id}} &&
	\w V_i\mid_{\w U_{ij}} \ar[r]^{f_i} \ar@/_10pt/[d]_{\w\nabla} \ar@/^10pt/[d]^{\w\nabla'_i} \ar@{..>}[dr] & 
	\w V_i'\mid_{\w U_{ij}} \ar[d]^{\w\nabla'}\\
	\w V\mid_{\w U_{ij}} \ar[r]^{f_j} & 
	\w V'\mid_{\w U_{ij}} && 
	\w V_i\otimes\Omega^1_{X/S}(\log D)\mid_{\w U_{ij}} \ar[r]^{f_i} & 
	\w V_i'\otimes\Omega^1_{X/S}(\log D) \mid_{\w U_{ij}}}
\end{equation*}
Then $f_j\mid_{\w U_{ij}} - f_i\mid_{\w U_{ij}} = f_j\circ(\mathrm{id} - f_{ij}') =f_j\circ\Delta f_{ij}$ and $\w \nabla'\circ f_i - f_i \circ \w\nabla = f_i\circ(\w\nabla'_i - \w\nabla)= f_i\circ\Delta \omega_i$. 
Thus $b((\w V,\w\nabla,\w\Fil),(\w V,\w\nabla,\w\Fil)') 
= \Big(
\big(\iota^{-1}_{(\w V,\w V')}(f_{j}\circ\Delta f_{ij})\big)_{(i,j)},
\big(\iota^{-1}_{(\w V,\w V')}(f_{i}\circ\Delta \omega_i)\big)_{i}
\Big) 
= \Big(
\big(\iota^{-1}(\Delta f_{ij})\big)_{(i,j)},
\big(\iota^{-1}(\Delta \omega_i)\big)_{i}
\Big)$.
\end{proof}

\begin{nota} For any lifting $(\w V,\w \nabla,\w \Fil)$ and any element $\varepsilon$ in $\mathbb H^1(\Fil^0\DR(\mEnd(V,\nabla)_{\mathcal I}))$, denote the $(\w V,\w\nabla,\w\Fil)'$ constructed in the proof of (c) by 
	\[(\w V,\w\nabla,\w\Fil) + \varepsilon.\] 
By definition, one has 
\[[b((\w V,\w\nabla,\w\Fil),(\w V,\w\nabla,\w\Fil)+\varepsilon)] = \varepsilon.\] 
For any lifting $(\w V,\w \nabla)$ and any element $\varepsilon$ in $\mathbb H^1(\DR(\mEnd(V,\nabla)_{\mathcal I}))$ one has similar notation $(\w V,\w\nabla) + \varepsilon$ and 
\[[b((\w V,\w\nabla),(\w V,\w\nabla)+\varepsilon)] = \varepsilon.\]  
\end{nota}

\subsection{Deformation of the Hodge filtration}
 In this section we will study the deformation theory of the Hodge filtration. For a filtered de Rham bundle over $(X,D)/S$, recall the complex $\mathscr C$ defined in (\ref{complex_(Fil)}) is 
 \begin{equation}
 0 \rightarrow \frac{\left(\mEnd(V)\right)_{\mathcal I}}{\left(\Fil^0 \mEnd(V)\right)_{\mathcal I}} \overset{\oln^\End}{\longrightarrow} \frac{\left(\mEnd(V)\right)_{\mathcal I}}{\left(\Fil^{-1} \mEnd(V)\right)_{\mathcal I}}\otimes\Omega_{X/S}^1(\log D) \overset{\oln^\End}{\longrightarrow}
 \frac{\left(\mEnd(V)\right)_{\mathcal I}}{\left(\Fil^{-2} \mEnd(V)\right)_{\mathcal I}}\otimes\Omega_{X/S}^2(\log D) \overset{\oln^\End}{\longrightarrow} \cdots
 \end{equation}
 
\begin{thm} \label{deformation_(Fil)} Let $(V,\nabla,\Fil)$ be a filtered de Rham bundle over $(X,D)/S$. Let $(\w V,\w\nabla)$ be lifting of the underlying de Rham bundle $(V,\nabla)$ over $(\w X,\w D)/\w S$. Then
\begin{itemize}
	\item[1).] the obstruction to lifting $\Fil$ to a Hodge filtration on $(\w V,\w\nabla)$ lies in $\mathbb H^1(\mathscr C)$;
	\item[2).] if $\Fil$ has a lifting, then the lifting set is an $\mathbb H^0(\mathscr C)$-torsor;
\end{itemize}
\end{thm}
In the rest of this subsection, we will always denote $\iota = \iota_{(\w V,\w V)}$.
\begin{lem}\label{local_data_(Fil)} Let $(V,\nabla,\Fil)$ be a filtered de Rham bundle over $(X,D)/S$. Let $\w V$ be lifting of vector bundle $V$ over $(\w X,\w D)/\w S$. Then 
	\begin{itemize}
		\item [(1).] there exists an affine open covering $\{U_i\}_{i\in I}$ such that  $\Fil^\ell V(U_i)$ is free over $\mathcal O_{X}(U_i)$ for all $i\in I$ and $\ell\in \mathbb Z$, and 
		\item [(2).] there exists local liftings $(\w\nabla_i,\w\Fil_i)$ of $(\nabla,\Fil)$ over $\w U_i$ such that $(\w\nabla_i,\w\Fil_i)$ satisfies Griffith transversality for each $i\in I$, and
		\item [(3).] there exists local liftings $f_{ij}\colon (\w V\mid_{\w U_i},\w \Fil_i)\mid_{\w U_{ij}} \rightarrow (\w V\mid_{\w U_j},\w \Fil_j)\mid_{\w U_{ij}}$ of $\mathrm{id}_{(V,\Fil)}$ over $\w U_{ij}$.
	\end{itemize}
	Overloading notation, we call such a triple $(\{U_i\}_{i},\{(\w\nabla_i,\w\Fil_i)\}_i,\{f_{ij}\}_{ij})$, whose existence is guaranteed by the lemma, a \emph{system of local data}.
\end{lem}
\begin{rmk}
	The connection $\w\nabla_i$ needs not to be integrable.
\end{rmk}
\begin{proof}
	The Lemma follows Lemma~\ref{local freeness_(V,Fil)} and Lemma~\ref{local_lifting_(V,Fil)}.
\end{proof}

Let $(V,\nabla,\Fil)$ be a filtered de Rham bundle over $(X,D)/S$. Let $(\w V,\w\nabla)$ be lifting of the underlying de Rham bundle $(V,\nabla)$ over $(\w X,\w D)/\w S$. Let $(\{U_i\}_{i},\{(\w\nabla_i,\w\Fil_i)\}_i,\{f_{ij}\}_{ij})$ be a system of local data given as in Lemma~\ref{local_data_(Fil)}. 
Then one defines an element
\begin{equation} 
\w c(\Fil)=\Big((\mathrm{id}-f_{ij})_{(i,j)},(\w\nabla_i - \w\nabla)_i\Big) 
\end{equation}
in 
\[\prod_{(i,j)\in I^2}
\left(\frac{\mEnd(\w V)}{\w\Fil_{ij}^{0}\mEnd(\w V)}\right)(\w U_{ij})
\times
\prod_{i\in I}\left(
\frac{\mEnd(\w V)}{\w\Fil_{ii}^{-1}\mEnd(\w V)}
\otimes
\Omega^1_{\w X/\w S}(\log \w D)
\right)(\w U_i).\]
Since $c(\Fil)\equiv 0\pmod{\mathcal I}$, one gets an element 
\begin{equation} \label{obs:Fil}
c(\Fil)=\Big(
\big(\iota^{-1}(\mathrm{id}-f_{ij})\big)_{(i,j)},
\big(\iota^{-1}(\w\nabla_i-\w\nabla)\big)_i
\Big) 
\end{equation}
in
\[\prod_{(i,j)\in I^2} \left(
\frac{\mEnd(V)_{\mathcal I}}{\Fil^{0} \mEnd(V)_{\mathcal I}}\right)(U_{ij})
\times
\prod_{i\in I}\left(
\frac{\mEnd(V)_{\mathcal I}}{\Fil^{-1} \mEnd(V)_{\mathcal I}}
\otimes
\Omega^1_{ X/ S}(\log  D) 
\right)(U_{i}).\]

\begin{lem} \label{1-cocycle_(Fil)} Let $(V,\nabla,\Fil)$ be a filtered de Rham bundle with $(\w V,\w \nabla)$ a lifting its underlying de Rham bundle. Let $(\{U_i\}_{i},\{(\w\nabla_i,\w\Fil_i)\}_i,\{f_{ij}\}_{ij})$ be a system of local data given as in Lemma~\ref{local_data_(Fil)}.
\begin{itemize}
	\item [(1).] The element $c(\Fil)$ defined in (\ref{obs:Fil}) is a $1$-cocycle in the Cech resolution of the  complex (\ref{complex_(Fil)}) with respect to the open covering $\{U_i\}_{i\in I}$ of $X$.
	\item [(2).] The class $[c(\Fil)] \in \mathbb H^1(\mathscr C)$ does not depend on the choice of  $(\{U_i\}_{i},\{(\w\nabla_i,\w\Fil_i)\}_i,\{f_{ij}\}_{ij})$.
\end{itemize}
\end{lem}

\begin{proof} The proof of (2) follow the same method of proof of (2) in Lemma~\ref{2-cocycle_(V,nabla,Fil)}. For the (1), we denote $\overline{\mathcal E}^k= \frac{\mEnd(V)_{\mathcal I}}{\Fil^{-k} \mEnd(V)_{\mathcal I}} \otimes \Omega^k_{ X/ S}(\log  D)$. Then the complex $\mathscr C$ may be rewritten as 
\[0\rightarrow 
\overline{\mathcal E}^0 \xrightarrow{\oln^\End}
\overline{\mathcal E}^1 \xrightarrow{\oln^\End}
\overline{\mathcal E}^2 \xrightarrow{\oln^\End}
\overline{\mathcal E}^3 \xrightarrow{\oln^\End}
 \cdots  \]
for all $k\geq 0$. For any $\tau=(i_1,\cdots,i_s)\in I^s$, denote $U_{(i_1,\cdots,i_s)} = \cap_{\ell =1}^s U_{i_\ell}$ and 
\[C^p(\mathcal U,\overline{\mathcal E}^k) = \prod_{\tau\in I^{p+1}} \overline{\mathcal E}^k(U_{\tau}).\]
Then the Cech resolution of the complex $\mathscr C$ with respect to the open covering $\mathcal U$ of $X$ is the total complex of the following double complex
\begin{equation}
\xymatrix{
	C^0(\mathcal U,\overline{\mathcal E}^0)
	\ar[d]^{\delta} \ar[r]^{\oln^\End} 
	& \red{ C^0(\mathcal U,\overline{\mathcal E}^1) }
	\ar[d]^{\delta} \ar[r]^{\oln^\End} 
	&   C^0(\mathcal U,\overline{\mathcal E}^2)
	\ar[d]^{\delta} \ar[r]^{\oln^\End}
	&  \cdots  \\
	\red{ C^1(\mathcal U,\overline{\mathcal E}^0) }
	\ar[d]^{\delta} \ar[r]^{\oln^\End} 
	& C^1(\mathcal U,\overline{\mathcal E}^1)
	\ar[d]^{\delta} \ar[r]^{\oln^\End} 
	&  C^1(\mathcal U,\overline{\mathcal E}^2)
	\ar[d]^{\delta} \ar[r]^{\oln^\End}
	&  \cdots  \\
	C^2(\mathcal U,\overline{\mathcal E}^0)
	\ar[d]^{\delta} \ar[r]^{\oln^\End} 
	&  C^2(\mathcal U,\overline{\mathcal E}^1)
	\ar[d]^{\delta} \ar[r]^{\oln^\End} 
	&  C^2(\mathcal U,\overline{\mathcal E}^2)
	\ar[d]^{\delta} \ar[r]^{\oln^\End} 
	&  \cdots  \\
	\vdots &  \vdots  & \vdots &   \\
}
\end{equation}
To show that
\[c(\Fil)=\Big(
\big(\iota^{-1}(\mathrm{id}-f_{ij})\big)_{(i,j)},
\big(\iota^{-1}(\w\nabla_i-\w\nabla)\big)_i
\Big)
\in 
C^1(\mathcal U,\overline{\mathcal E}^0) \times 
C^0(\mathcal U,\overline{\mathcal E}^1)\] 
forms a cocycle, one needs to check
\begin{equation*}
\left\{
\begin{array}{lclll}
0 
& = & 
\delta\left( \big(\iota^{-1}(\mathrm{id}-f_{ij})\big)_{(i,j)} \right) 
& \in C^2(\mathcal U,\overline{\mathcal E}^0) \\
\oln^\End\left( \big(\iota^{-1}(\mathrm{id}-f_{ij})\big)_{(i,j)} \right) 
& = & 
\delta\left( \big(\iota^{-1}(\w\nabla_i-\w\nabla)\big)_i \right)  
& \in C^1(\mathcal U,\overline{\mathcal E}^1) \\
\oln^\End\left( \big(\iota^{-1}(\w\nabla_i-\w\nabla)\big)_i \right)  
& = & 
0 
& \in C^0(\mathcal U,\overline{\mathcal E}^2) \\
\end{array}
\right.
\end{equation*}
This is equivalent to checking
\begin{itemize}
	\item[1a).] 
	$ (\mathrm{id}-f_{jk}) - (\mathrm{id}-f_{ik}) + (\mathrm{id}-f_{ij}) \equiv 0 \pmod{\Fil^0_{ik}\mEnd(\w V)(\w U_{ijk})}$ for each $(i,j,k)\in I^3$.
	\item[1b).] 
	$\w\nabla_{ij}(\mathrm{id}-f_{ij}) 
	\equiv 
	(\w\nabla_j - \w\nabla) - (\w\nabla_i - \w\nabla) 
    \pmod{\big(\Fil^{-1}_{ij}\mEnd(\w V)\otimes \Omega_{\w X/\w S}^1(\log \w D)\big)(\w U_{ij})}$ 
	for each $(i,j)\in I^2$,
	\item[1c).] $\w\nabla_{ii}(\w\nabla_i - \w\nabla)  \equiv 0 \pmod{\big(\Fil_{ii}^{-2}\mEnd(\w V) \otimes \Omega_{\w X/\w S}^2(\log \w D)\big)(\w U_i)}$ for each $i\in I$.
\end{itemize}

Denote $\alpha_{ij}= \mathrm{id}_{\w V}-f_{ij}$. Then
\[f_{ik} - f_{jk} \circ f_{ij} = \alpha_{jk} - \alpha_{ik} + \alpha_{ij} - \alpha_{jk} \circ \alpha_{ij}.\]
Since $\alpha_{ij}\equiv 0\equiv \alpha_{jk}\pmod{\mathcal I}$ and $\mathcal I^2=0$, $\alpha_{jk}\circ \alpha_{ij} =0$. By definition $f_{ij}\in \Fil^0_{ij}\mEnd(\w V)$, thus $f_{ik} - f_{jk}\circ f_{ij} \in \Fil^0_{ik}\mEnd(\w V)$ by Lemma~\ref{main_lemma} and 1a) follows.

Note that 
\begin{equation*}
\begin{split}
\w\nabla_{ij}(\mathrm{id}-f_{ij}) = & \w\nabla_j\circ (\mathrm{id}-f_{ij}) - (\mathrm{id}-f_{ij})\circ \w\nabla_i\\
= & (\w\nabla_j - \w\nabla) -(\w\nabla_i - \w\nabla) - (\w\nabla_j\circ f_{ij} - f_{ij}\circ \w\nabla_i)
\end{split}
\end{equation*}
By lemma~\ref{main_lemma}, $(\w\nabla_j\circ f_{ij} - f_{ij}\circ \w\nabla_i) \in \Fil^{-1}_{ij}\mEnd(\w V) \otimes \Omega_{\w X/\w S}^1(\log \w D)$ and 1b) follows.

Note that
\begin{equation*}
\begin{split}
\w\nabla_{ii}(\w\nabla_i - \w\nabla\mid_{\w U_i})  
= & 
\w\nabla_i\circ (\w\nabla_i - \w\nabla\mid_{\w U_i}) + (\w\nabla_i - \w\nabla\mid_{\w U_i})\circ \w\nabla_i\\
= & (\w\nabla_i-\w\nabla)^2 +\w\nabla_i^2 -\w\nabla^2\\
\end{split}
\end{equation*}
Since both $\w\nabla_i-\w\nabla \equiv 0\pmod{\mathcal I}$, $(\w\nabla_i-\w\nabla)^2=0$. Since $\w\nabla$ is integrable, $\w\nabla^2=0$. Thus 1c) follows Lemma~\ref{main_lemma}.
\end{proof}

\begin{lem}\label{local_isomorphism_(Fil)}
Let $(V,\nabla,\Fil)$ be a filtered de Rham bundle over $(X,D)/S$. Let $(\w V,\w\nabla)$ be a lifting of the underlying de Rham bundle $(V,\nabla)$ over $(\w X,\w D)/\w S$. Suppose there are two Hodge filtrations $\w\Fil$ and $\w\Fil'$ on $(\w V,\w\nabla)$ which lift  the $\Fil$. Then
\begin{itemize}
	\item [(1).] there exists an open affine covering  $\{U_i\}_{i\in I}$ such that $\w \Fil^\ell \w V(\w U_i)$ and $\w \Fil'^\ell \w V(\w U_i)$ are free over $\mathcal O_{\w X}(\w U_i)$ for all $i\in I$ and $\ell\in \mathbb Z$, and 
	\item [(2).] there exists a (local) isomorphism $f_i\colon (\w V,\w\Fil)\mid_{\w U_i} \rightarrow (\w V,\w\Fil')\mid_{\w U_i}$ lifting the identity $\mathrm{id}_{V\mid_{U_i}}$ for each $i\in I$.
\end{itemize}
Overloading notation, we call a pair  $(\{U_i\}_{i\in I},\{f_i\}_i)$, whose existence is guaranteed by the lemma, a \emph{system of local data}.
\end{lem}
\begin{proof}
	The lemma follows Lemma~\ref{local freeness_(V,Fil)} and Lemma~\ref{local_lifting_(V,Fil)}.
\end{proof}

Let $(V,\nabla,\Fil)$ be a filtered de Rham bundle over $(X,D)/S$. Let $(\w V,\w\nabla)$ be lifting of the underlying de Rham bundle $(V,\nabla)$ over $(\w X,\w D)/\w S$. Suppose there are two Hodge filtrations $\w\Fil$ and $\w\Fil'$ on $(\w V,\w\nabla)$ which lift the $\Fil$. Let $(\{U_i\}_{i\in I},\{f_i\}_i)$ be a system of local data given as in Lemma~\ref{local_isomorphism_(Fil)}. Then 
\begin{equation} \label{torsor_(Fil)}
 \w b(\w\Fil,\w\Fil') = \Big(\mathrm{id}-f_{i}\Big)_{i} \in  \prod_{i\in I}\left( \frac{\mEnd(\w V)}{\Fil^{0}_{(\w\Fil,\w\Fil')}\mEnd(\w V)}\right)(\w U_i).
\end{equation}
Since $\w b(\w\Fil,\w\Fil')\equiv 0\pmod{\mathcal I}$, one gets an element 
\begin{equation} 
b(\w\Fil,\w\Fil') = \Big(\iota^{-1}(\mathrm{id}-f_{i})\Big)_{i} \in \prod_{i\in I} \left(\frac{\mEnd(V)} {\Fil^{0}\mEnd(V)}\right)(U_i).
\end{equation}
where $\iota = \iota_{(\w V,\w V)}$.
\begin{lem} \label{0-cocycle_(Fil)} Let $(\w V,\w \nabla)$ be a lifting of the underlying de Rham bundle of a filtered de Rham bundle $(V,\nabla,\Fil)$. Let $\w\Fil$ and $\w\Fil'$ be two Hodge filtrations on $(\w V,\w\nabla)$ which lift the $\Fil$. Then
\begin{itemize}	
	\item [(1).] The element $b(\w\Fil,\w\Fil')$ defined in (\ref{torsor_(Fil)}) is a $0$-cocycle in the Cech resolution of the  complex (\ref{complex_(Fil)}) with respect to the open covering $\{U_i\}_{i\in I}$ of $X$.
	\item [(2).] The class $[b(\w\Fil,\w\Fil')] \in \mathbb H^0(\mathscr C)$ does not depend on the choice of $(\{U_i\}_{i\in I},\{f_i\}_i)$.
\end{itemize}
\end{lem}
\begin{proof}The proof of (2) follows the same method of proof of (2) in Lemma~\ref{2-cocycle_(V,nabla,Fil)}. For (1), to proof 
\[b(\w\Fil,\w\Fil') = \Big(\iota^{-1}(\mathrm{id}-f_{i})\Big)_{i}\]	
forms a $1$-cocycle, we only need to check 
\begin{equation*}
 \delta\left(\Big(\iota^{-1}(\mathrm{id}-f_{i})\Big)_{i}\right)= 0 \in  C^1(\mathcal U,\overline{\mathcal E}^0) \quad \text{ and } \quad \oln^{\End}\left(\Big(\iota^{-1}(\mathrm{id}-f_{i})\Big)_{i}\right)=0\in   C^0(\mathcal U,\overline{\mathcal E}^1).
\end{equation*}  
This is equivalent to check
\begin{itemize}
	\item[1a).] $(\mathrm{id}-f_{j})-(\mathrm{id}-f_{i}) \in \Fil^0_{(\w\Fil,\w\Fil')}\mEnd(\w V)(\w U_{ij})$.
	\item[1b).] $\w\nabla_i\circ (\mathrm{id}-f_{i}) - (\mathrm{id}-f_{i})\circ \w\nabla_i \in \Big(\Fil^{-1}_{(\w\Fil,\w\Fil')}\mEnd(\w V)\otimes\Omega^1_{\w X/\w S}(\log \w D)\Big)(\w U_i)$.
\end{itemize}
This follows Lemma~\ref{main_lemma}.
\end{proof}

\begin{proof}[Proof of Theorem~\ref{deformation_(Fil)}]
 According Lemma~\ref{1-cocycle_(Fil)} and Lemma~\ref{0-cocycle_(Fil)} one only need to show that 
\begin{itemize}
	\item[(a)] if the Hodge filtration $\Fil$ has a lifting $\w\Fil$, then $[c(\Fil)]=0$;
	\item[(b)] if $[c(\Fil)]=0$, then the Hodge filtration $\Fil$ is liftable;
	\item[(c)] Let $\w\Fil$ be a lifting the Hodge filtration $\Fil$. For any element $\epsilon$ in $\mathbb H^0(\mathscr C)$, there exists another lifting $\w\Fil'$ such that $\epsilon = b( \w\Fil,\w\Fil')$.
\end{itemize}
The proof of (a) is trivial: if we simply choose the system of local data $\big(\{U_i\}_i,\{f_{i}\}_{i}\big)$ with $f_{i} = \mathrm{id}_{\w V\mid_{\w U_{i}}}$, then $c(\Fil) =0$. 

\paragraph{Proof of (b).} Since 
$c(\Fil)=\Big(
\big(\iota^{-1}(\mathrm{id}-f_{ij})\big)_{(i,j)},
\big(\iota^{-1}(\w\nabla_i-\w\nabla)\big)_i
\Big) $
is a coboundary, there exists an element 
\[\big(\iota^{-1}(\Delta f_{i})\big)_i \in \prod_{i\in I} \left(
\frac{\mEnd(V)_{\mathcal I}} {\Fil^{0}_{ii}\mEnd(V)_{\mathcal I}}
\right)(U_{i})\]
with  $\Delta f_i \in \mathcal I\cdot \mEnd(\w V)$ such that 
\begin{itemize}
	\item[(b1)] $\mathrm{id} - f_{ij} \equiv \Delta f_{j} - \Delta f_{i} \mod \Big(\Fil^0_{ij}\mEnd(\w V)\Big)(\w U_{ij})$.
	\item[(b2)] $ \w\nabla_i-\w\nabla \equiv \w\nabla \circ \Delta f_i - \Delta f_i \circ \w\nabla \mod  \Big( \Fil^{-1}_{ii}\mEnd(\w V) \otimes\Omega^1_{\w X/\w S}(\log \w D) \Big)(\w U_{i})$. 
\end{itemize}
Consider the new $\w \Fil_i' = (\mathrm{id}_{\w V\mid_{\w U_i}}  - \Delta f_i)^*(\w Fil_i)$ on $\w V_i$. Then $\mathrm{id}_{\w V\mid_{\w U_i}}  - \Delta f_i \in \Fil^0_{\Fil'_i,\Fil_i}\mEnd(\w V) $
\begin{equation*}
\xymatrix@C=3cm{
\left(\w V\mid_{\w U_{ij}},\w\Fil'_{i}\right) 
\ar@/_3pt/[r]_{f_{ij} + \Delta f_j - \Delta f_i} 
\ar@/^3pt/[r]^{\mathrm{id}} 
\ar[d]_{\mathrm{id}-\Delta f_i} 
&  \left(\w V\mid_{\w U_{ij}},\w\Fil'_{j}\right) \ar[d]^{\mathrm{id}-\Delta f_j} \\
\left(\w V\mid_{\w U_{ij}},\w\Fil_{i}\right) \ar[r]^{f_{ij}}  
&  \left(\w V\mid_{\w U_{ij}},\w\Fil_{j}\right)\\
}
\end{equation*}
Thus $f_{ij} + \Delta f_j - \Delta f_i = (\mathrm{id}-\Delta f_j)^{-1}\circ f_{ij} \circ (\mathrm{id}-\Delta f_i) \in \Fil^0_{\Fil'_i,\Fil'_j}\mEnd(\w V)$ by Lemma~\ref{main_lemma}. Thus $\mathrm{id} \in \Fil^0_{\Fil'_i,\Fil'_j}\mEnd(\w V)$ by (b1), and $\w\Fil'_i$, $\w\Fil'_j$  coincide over $\w U_{ij}$ by Lemma~\ref{main_lemma}. Hence one may glue local filtrations $\{\w \Fil'_i\}$ into a filtration $\w \Fil$ of $\w V$.

Since $(\w\nabla_i,\w\Fil_i)$ satisfies Griffith tranversality, $((\mathrm{id}_{\w V\mid_{\w U_i}} - \Delta f_i)^*\w\nabla_i,\w\Fil_i')$ also satisfies Griffith transversality. Since  $(\mathrm{id}_{\w V\mid_{\w U_i}} - \Delta f_i)^*\w\nabla_i = \w\nabla_i - \w\nabla \circ \Delta f_i + \Delta f_i \circ \w\nabla$,  (b2) implies that $(\w\nabla\mid_{\w U_i},\w\Fil_i')$ also satisfies Griffith transversality by Lemma~\ref{diff_conn}. Thus $(\w V,\w\nabla,\w\Fil)$ is a filtered de Rham bundle.  

\paragraph{Proof of (c).} Denote $\iota=\iota_{\w V,\w V}$, $(\w\nabla_i,\w\Fil_i) =(\w\nabla,\w\Fil)\mid_{\w U_i}$ and $f_{ij} = \mathrm{id}_{\w V\mid_{\w U_{ij}}}$. Let \[\big(\iota^{-1}(\Delta f_{i})\big)_i \in \prod_{i\in I} \left(
\frac{\mEnd(V)_{\mathcal I}} {\Fil^{0}_{ii}\mEnd(V)_{\mathcal I}}
\right)(U_{i})\]	
be a representation of $\epsilon$ with $\Delta f_i \in \mathcal I\cdot \mEnd(\w V)$. Then 
\begin{itemize}
	\item[(c1)] $\mathrm{id} - f_{ij} \equiv 0 \equiv  \Delta f_{j} - \Delta f_{i}  \mod \Big(\Fil^0_{ij}\mEnd(\w V)\Big)(\w U_{ij})$.
	\item[(c2)] $ \w\nabla_i-\w\nabla \equiv 0 \equiv \w\nabla \circ \Delta f_i - \Delta f_i \circ \w\nabla \mod  \Big( \Fil^{-1}_{ii}\mEnd(\w V) \otimes\Omega^1_{\w X/\w S}(\log \w D) \Big)(\w U_{i})$. 
\end{itemize}
By the proof of (b), one constructs a new filtration $\w\Fil'$ such that $(\w V,\w\nabla,\w\Fil')$ forms a filtered de Rham bundle with local isomorphisms 
\[\mathrm{id} - \Delta f_i\colon (\w V,\w\Fil')\mid_{\w U_i} \rightarrow (\w V,\w\Fil)\mid_{\w U_i}.\]
Thus $b(\Fil,\Fil') = \big(\iota^{-1}(\mathrm{id} - (\mathrm{id} - \Delta f_i))\big)_i =\big(\iota^{-1}(\Delta f_i)\big)_i$.
\end{proof}

\section{Deformations of (graded) Higgs bundles}

In this section we will study the deformation theory of graded Higgs bundles. The main result is Theorem~\ref{deformation_(E,theta,Gr)}.

Let $S$, $\w{S}$, $\mathfrak a$, $\w{X}$, $\w D$, $X$, $D$ and $\mathcal I$ be as in Setup \ref{setup}.  Let $(E,\theta)$ be a (logarithmic) Higgs bundle over $(X,D)/S$.  The pair $\mEnd(E,\theta)$ is naturally a (logarithmic) Higgs bundle. Thus one has the following Higgs complex 
\begin{equation}
\DR(\mEnd(E,\theta)): \qquad 0 \rightarrow \mEnd(E)\overset{\theta^\End}{\longrightarrow} \mEnd(E)\otimes\Omega_{X/S}^1(\log D)\overset{\theta^\End}{\longrightarrow}  \mEnd(E)\otimes\Omega_{X/S}^2(\log D) \overset{\theta^\End}{\longrightarrow} \cdots
\end{equation}
Tensoring with the ideal sheaf $\mathcal I$, one gets complex
\begin{equation} \label{complex_(E,theta)}
\DR(\mEnd(E,\theta))_{\mathcal I}: \qquad 0 \rightarrow \mEnd(E)_{\mathcal I}\overset{\theta^\End}{\longrightarrow} \mEnd(E)_{\mathcal I}\otimes\Omega_{X/S}^1(\log D) \overset{\theta^\End}{\longrightarrow}  \mEnd(E)_{\mathcal I}\otimes\Omega_{X/S}^2(\log D) \overset{\theta^\End}{\longrightarrow} \cdots
\end{equation}

Suppose $(E,\theta,\Gr)$ is a graded Higgs bundle. Then $\Gr$ induces a grading structure on the complex $\DR(\mEnd(E,\theta))\otimes \mathcal I$ with
\begin{equation}  \label{complex_(E,theta,Fil)}
\mathrm{Gr}^\ell \DR(\mEnd(E,\theta))_{\mathcal I}: \qquad 0 \rightarrow \mathrm{Gr}^\ell\mEnd(E)_{\mathcal I} \overset{\theta^\End}{\longrightarrow} \mathrm{Gr}^{\ell-1}\mEnd(E)_{\mathcal I}\otimes\Omega_{X/S}^1(\log D)\overset{\theta^\End}{\longrightarrow} \cdots
\end{equation}

\begin{thm}\label{deformation_(E,theta,Gr)} Let $(E,\theta,\Gr)$ be a graded Higgs bundle over $(X,D)/S$. Then
	\begin{itemize}
		\item[1).] the obstruction to lifting $(E,\theta,\Gr)$ to a graded Higgs bundle over $(\w X,\w D)/\w S$
		 lies in $\mathbb H^2\big(\mathrm{Gr}^0 \DR(\mEnd(E,\theta))_{\mathcal I}\big)$;
		\item[2).] if $(E,\theta,\Gr)$ has a graded lifting $(\widetilde{E},\widetilde{\theta},\w\Gr)$, then the lifting set is an $\mathbb H^1\big(\mathrm{Gr}^0 \DR(\mEnd(E,\theta))_{\mathcal I}\big)$-torsor;
		\item[3).] the infinitesimal automorphism group of $(\widetilde{E},\widetilde{\theta},\w\Gr)$ over $(E,\theta,\Gr)$ is $\mathbb H^0\big(\mathrm{Gr}^0 \DR(\mEnd(E,\theta))_{\mathcal I}\big)$.
	\end{itemize}
\end{thm}

\begin{proof}[Proof of Theorem~\ref{deformation_(E,theta,Gr)}]
The proof is almost the same as that of Theorem~\ref{deformation_(V,nabla,Fil)}. One only needs to replace connections with Higgs fields and filtrations with graded structures. We give an outline; the details are left to readers.

\textbf{First step.} As in Lemma~\ref{local_data_(V,nabla,Fil)}, one may find an open covering $\{U_i\}_{i\in I}$, local liftings $(\w E_i, \w\Gr_i)$ of the graded vector bundle $(E,\Gr)\mid_{U_i}$, local liftings $\w\theta_i$ (as an $\mathcal O_{\w X}$-linear map) of $\theta\mid_{U_i}$ with $\w\theta_i(\w\Gr^\ell \w E)\subset \w \Gr^{\ell -1}\w E \otimes \Omega^1_{\w X/\w S}(\log \w D)$ and local isomorphisms $f_{ij}\colon (\w E_i,\w\theta_i,\w\Gr_i)\mid_{\w U_{ij}} \rightarrow (\w E_i,\w\theta_i,\w\Gr_i)\mid_{\w U_{ij}}$ of graded vector bundles with graded linear maps. Note that here $\w\theta_i$ is only $\mathcal O_{\w X}$-linear and NOT integrable. This is the \emph{system of local data}.
                                                                                
\textbf{Second step.} As in Lemma~\ref{2-cocycle_(V,nabla,Fil)}, one shows that 
\begin{equation}\label{obs:c(E,theta,Gr)}
\w c(E,\theta,\Gr):= \Big(-(f_{jk}\circ f_{ij}-f_{ik}), (\w \theta_j\circ f_{ij} - f_{ij}\circ\w\theta_i), (\w\theta_i\circ\w\theta_i)\Big)
\end{equation}
defines an element $c(E,\theta,\Gr)$ in $\mathbb H^2\big(\mathrm{Gr}^0 \DR(\mEnd(E,\theta))_{\mathcal I}\big)$, which does not depend on the choice of the system of local data $(\{U_i\}_{i\in I},\{(\w E_i,\w\theta_i,\w\Gr_i)\}_i,\{f_{ij}\}_{ij}\})$.

\textbf{Third step.} Let $(\w E,\w\theta,\w\Gr)$ and $(\w E,\w\theta,\w\Gr)'$  be two liftings of the graded Higgs bundle $(E,\theta,\Gr)$. As in Lemma~\ref{local_isomorphism_(V,nabla,Fil)}, one find open covering $\{U_i\}_{i\in I}$ and local isomorphisms $f_i\colon (\w E,\w \Gr)\mid_{\w U_i} \rightarrow (\w E,\w \Gr)\mid_{\w U_i}$ for each $i\in I$.

\textbf{Forth step.} As in Lemma~\ref{1-cocycle_(V,nabla,Fil)}, one shows that  
\begin{equation}\label{torsor_(E,theta,Gr)}
\w b((\w E,\w\theta,\w\Gr),(\w E,\w\theta,\w\Gr)'):=((f_j-f_i),(\w\theta'\circ f_{i}-f_i\circ\w\theta))
\end{equation}
defines an element $b((\w E,\w\theta,\w\Gr),(\w E,\w\theta,\w\Gr)')$ in $\mathbb H^1\big(\mathrm{Gr}^0 \DR(\mEnd(E,\theta))_{\mathcal I}\big)$, which does not depend on the choice of the system of local data $(\{U_i\}_{i\in I},\{f_{i}\}_{i}\})$.

\textbf{Fifth step.}  Let $(\w E,\w\theta,\w\Gr)$ be a lifting of the graded Higgs bundle $(E,\theta,\Gr)$. As in Lemma~\ref{0-cocycle_(V,nabla,Fil)}, one shows that $\mathrm{id} + \epsilon $ is an automorphism of $(\w E,\w\theta,\w\Gr)$ for any $\epsilon \in \mathbb H^0\big(\mathrm{Gr}^0 \DR(\mEnd(E,\theta))_{\mathcal I}\big)$ and any automorphism of $(\w E,\w\theta,\w\Gr)$ which lifts the identity map is of this form.

\textbf{Sixth step.} As in the proof of Theorem~\ref{deformation_(V,nabla,Fil)}, one shows that 
\begin{itemize}
	\item[(a)] if the graded Higgs bundle $(E,\theta,\Gr)$ has a lifting $(\w E,\w\theta,\w\Gr)$, then $[c(E,\theta,\Gr)]=0$;
	\item[(b)] if $[c(E,\theta,\Gr)]=0$, then the graded Higgs bundle $(E,\theta,\Gr)$ is liftable;
	\item[(c)] Let $(\w E,\w\theta,\w\Gr)$ be a lifting of the graded Higgs bundle $(E,\theta,\Gr)$. For any element $\epsilon$ in $\mathbb H^1(\Gr^0\DR(\mEnd(E,\theta))\otimes I)$, there exists another lifting $(\w E,\w\theta,\w\Gr)'$ such that $\epsilon = b((\w E,\w\theta,\w\Gr),(\w E,\w\theta,\w\Gr)')$.
\end{itemize}
\end{proof}

Ignore all graded structure and follow the similar method, one shows the following classical result.
\begin{prop}\label{deformation_(E,theta)} Let $(E,\theta)$ be a Higgs bundle over $(X,D)/S$. Then
	\begin{itemize}
		\item[1).] the obstruction to lifting $(E,\theta)$ to a Higgs bundle over $(\w X,\w D)/\w S$ lies in $\mathbb H^2\big(\DR(\mEnd(E,\theta))_{\mathcal I}\big)$;
		\item[2).] if $(E,\theta)$ has lifting $(\widetilde{E},\widetilde{\theta})$, then the lifting set is an $\mathbb H^1\big(\DR(\mEnd(E,\theta))_{\mathcal I}\big)$-torsor;
		\item[3).] the infinitesimal automorphism group of $(\widetilde{E},\widetilde{\theta})$ over $(E,\theta)$ is $\mathbb H^0\big(\DR(\mEnd(E,\theta))_{\mathcal I}\big)$.
	\end{itemize}
\end{prop}

\section{Uniqueness of Hodge filtration}
Let $S$, $\w{S}$, $\mathfrak a$, $\w{X}$, $\w D$, $X$, $D$, $\mathcal I$ be as in Setup~\ref{setup}. In this section, the main result is Theorem~\ref{thm: corollary of E1 degenarate}.

Let $(V,\nabla,\Fil)$ be a logarithmic filtered de Rham bundle on $(X,D)/S$. Let $(E,\theta,\Gr)$ be the associated graded, which is a logarithmic graded Higgs bundle on $(X,D)/S$ as in Example~\ref{exa1}. Recall the hypercohomology spectral sequence attached to a filtered complex in \cite[1.4.5]{Del71}; one therefore has the following \emph{Hodge-de Rham spectral sequence}:
\begin{equation} \label{spectral sequence}
\mathbb H^{p+q}(\mathrm{Gr}^p \DR(\mEnd(E,\theta)_{\mathcal I})) \Rightarrow \mathbb H^{p+q}(\DR(\mEnd(V,\nabla)_\mathcal I)).
\end{equation}

Recall the complex $\mathscr C$ defined in~(\ref{complex_(Fil)}). One has a short exact sequence of complexes
\[0\rightarrow \Fil^0 \DR(\mEnd(V,\nabla)_\mathcal I) \overset{\iota}{\longrightarrow} 
\DR(\mEnd(V,\nabla)_\mathcal I)
\overset{\pi}{\longrightarrow}
 \mathscr C \rightarrow 0,\]
which induces a long exact sequence
\begin{equation} \label{long exact sequence}
\begin{split}
0 & \rightarrow 
\mathbb H^0\Big(\Fil^0 \DR(\mEnd(V,\nabla)_\mathcal I)\Big) 
\overset{\iota}{\longrightarrow}
\mathbb H^0\Big(\DR(\mEnd(V,\nabla)_\mathcal I) \Big)
\overset{\pi}{\longrightarrow}
\mathbb H^0\Big(\mathscr C \Big)\\
& \overset{\delta}{\longrightarrow} 
\mathbb H^1\Big(\Fil^0 \DR(\mEnd(V,\nabla)_\mathcal I)\Big) 
\overset{\iota}{\longrightarrow}
\mathbb H^1\Big(\DR(\mEnd(V,\nabla)_\mathcal I) \Big)
\overset{\pi}{\longrightarrow}
 \mathbb H^1\Big(\mathscr C \Big)\\
&  \overset{\delta}{\longrightarrow}  \mathbb H^2\Big(\Fil^0 \DR(\mEnd(V,\nabla)_\mathcal I)\Big) 
\overset{\iota}{\longrightarrow}
\mathbb H^2\Big(\DR(\mEnd(V,\nabla)_\mathcal I) \Big)
\overset{\pi}{\longrightarrow} \cdots\\
\end{split}
\end{equation}

Denote by  $\mathrm{Def}_{(V,\nabla)}(\widetilde X,\widetilde D)$ the set of deformations the of de Rham bundle to $(\widetilde X,\widetilde D)$:
\[\mathrm{Def}_{(V,\nabla)}(\widetilde X,\widetilde D) := \left\{ \text{ de Rham bundle $(\w V,\w\nabla)$ over $(\w X,\w D)/\w S$ with }  (\w V,\w\nabla)\mid_X = (V,\nabla)  \right\}/\sim ,\]                                                                
which is an $\mathbb H^{1}(\DR(\mEnd(V,\nabla)_\mathcal I))$-torsor space if it is not empty. Thus for any element $(\w V_{n+1},\w \nabla_{n+1}) \in \mathrm{Def}_{(V_{n},\nabla_{n})}(X_{n+1},D_{n+1})$ and any element $\varepsilon \in \mathbb H^{1}(\DR(\mEnd(V,\nabla)_\mathcal I))$, there exists a unique element  
$(\w V_{n+1}',\w \nabla_{n+1}') \in \mathrm{Def}_{(V_{n},\nabla_{n})}(X_{n+1},D_{n+1})$
such that 
\[b((\w V_{n+1},\w \nabla_{n+1}), (\w V_{n+1}',\w \nabla_{n+1}')) = \varepsilon\]
where $b$ is defined in Lemma~\ref{1-cocycle_(V,nabla,Fil)}. We will simply denote 
\begin{equation}
(\w V_{n+1},\w \nabla_{n+1}) + \varepsilon := (\w V_{n+1}',\w \nabla_{n+1}'). 
\end{equation}

Similarly denote by $\mathrm{Def}_{(V,\nabla,\Fil)}(\widetilde X,\widetilde D)$ the set of deformations of the filtered de Rham bundle
\[\mathrm{Def}_{(V,\nabla,\Fil)}(\widetilde X,\widetilde D) := \left\{ \text{ filtered de Rham bundle $(\w V,\w\nabla,\w\Fil) $over $(\w X,\w D)/\w S$ with }  (\w V,\w\nabla,\w\Fil)\mid_X = (V,\nabla,\Fil)  \right\}/\sim.\] 
One also has the notation $(\w V_{n+1},\w \nabla_{n+1},\w\Fil_{n+1}) + \varepsilon$ for any $(\w V_{n+1},\w \nabla_{n+1},\Fil) \in \mathrm{Def}_{(V,\nabla,\Fil)}(\widetilde X,\widetilde D)$ and any $\varepsilon \in \mathbb H^{1}(\Fil^0\DR(\mEnd(V,\nabla)_\mathcal I))$.


\begin{lem}\label{prop:splitting}
	Suppose that the Hodge-de Rham spectral sequence (\ref{spectral sequence}) degenerates at $E_1$. Then the long exact sequence~\ref{long exact sequence}
		splits into short exact sequences	
		\[0
		\rightarrow \mathbb H^k\Big(\Fil^0 \DR(\mEnd(V,\nabla)_\mathcal I)\Big) 
		\longrightarrow \mathbb H^k\Big(\DR(\mEnd(V,\nabla)_\mathcal I) \Big)
		\longrightarrow \mathbb H^k\Big(\mathscr C \Big) \rightarrow 0.\]
\end{lem}  
\begin{proof} By the definition of $E_1$ degenaration, $\mathbb H^k\Big(\Fil^0 \DR(\mEnd(V,\nabla)_\mathcal I)\Big) 
		\rightarrow \mathbb H^k\Big(\DR(\mEnd(V,\nabla)_\mathcal I) \Big)$ is an injective map. Thus the lemma follows.
\end{proof}

\begin{thm}\label{thm: corollary of E1 degenarate}
	Suppose that the spectral sequence (\ref{spectral sequence}) degenerates at $E_1$. Then
	\begin{itemize}
		\item[(1).] If $(V,\nabla)$ is liftable, then $(V,\nabla,\Fil)$ is also liftable.
		\item[(2).] the map $\mathrm{Def}_{(V,\nabla,\Fil)}(\w X,\w D) \rightarrow \mathrm{Def}_{(V,\nabla)}(\w X,\w D)$ is injective. In other words, if $(\w V,\w \nabla)$ is a lifting of the underlying de Rham bundle of $(V,\nabla,\Fil)$, then for two liftings $\w\Fil$ and $\w\Fil'$ of the Hodge filtration $\Fil$ there exists an isomorphism
		\[(\w V,\w \nabla,\w \Fil)\overset{\simeq}{\longrightarrow} (\w V,\w \nabla,\w \Fil')\]
		lifting the identity over $S$.
	\end{itemize}
\end{thm} 
\begin{rmk}
	(1) in Theorem~\ref{thm: corollary of E1 degenarate} means that if $(V,\nabla)$ has a lifting $(\widetilde{V},\widetilde{\nabla})$ then  $(V,\nabla,\Fil)$ has some lifting $(\widetilde{V}',\widetilde{\nabla}',\widetilde{\Fil}')$. In general, the underlying de Rham bundle of $(\widetilde{V}',\widetilde{\nabla}',\widetilde{\Fil}')$ is not equal $(\widetilde{V},\widetilde{\nabla})$ and there exists obstruction to lift the Hodge filtration to $(\widetilde{V},\widetilde{\nabla})$.
\end{rmk}

\begin{lem}\label{torsor map} Let $(V,\nabla,\Fil)$ be a filtered de Rham bundle. Then
	\begin{itemize}
		\item[(1).] $\iota(c(V,\nabla,\Fil)) = c(V,\nabla)$;
		\item[(3).] Suppose $c(V,\nabla,\Fil)=0$. Let $(\w V,\w\nabla,\w\Fil)$ and $(\w V,\w\nabla,\w\Fil)'$ be two liftings of the filtered de Rham bundle $(V,\nabla,\Fil)$. Then
		\[\iota(b((\w V,\w\nabla,\w\Fil),(\w V,\w\nabla,\w\Fil)')) = b((\w V,\w\nabla),(\w V,\w\nabla)').\]
		In other words, for any given lifting $(\w V,\w\nabla,\w\Fil)\in \mathrm{Def}_{(V,\nabla,\Fil)}(\w X,\w D)$ one has the following commutative diagram
		\begin{equation}
		\xymatrix{
			\mathrm{Def}_{(V,\nabla,\Fil)}(\widetilde X,\widetilde D)\ar[r]
			\ar@/_10pt/[d]_{(\w V,\w\nabla,\w\Fil)'\mapsto b((\w V,\w\nabla,\w\Fil),(\w V,\w\nabla,\w\Fil)')}
			& \mathrm{Def}_{(V,\nabla)}(\widetilde X,\widetilde D) 
			\ar@/^10pt/[d]^{(\w V,\w\nabla)'\mapsto b((\w V,\w\nabla),(\w V,\w\nabla)')}\\
			\mathbb H^1\Big(\Fil^0 \DR(\mEnd(V,\nabla)_\mathcal I)\Big) \ar[r]^\iota \ar@/_10pt/[u]_{\epsilon \mapsto (\w V,\w\nabla) +\epsilon} &
			\mathbb H^1\Big(\DR(\mEnd(V,\nabla)_\mathcal I) \Big) \ar@/^10pt/[u]^{\epsilon \mapsto (\w V,\w\nabla) +\epsilon}\\}
		\end{equation}	
	\end{itemize}
\end{lem}

\begin{proof}
 This follows the definition of $c$ and $b$.
\end{proof}

\begin{proof}[Proof of Theorem~\ref{thm: corollary of E1 degenarate}]
By Lemma~\ref{prop:splitting}, the $\iota$'s are injective. Thus lemma follows Lemma~\ref{torsor map}.
\end{proof}

\section{Lifting a Higgs-de Rham flow.}\label{section:lifting_HDF}

In this section we will study the deformation theory of a 1-periodic Higgs-de Rham flow. The main result is Theorem~\ref{main_thm}.

Let $k$ be a perfect field of odd characteristic $p$,  and set $S_m = \Spec (W_m(k))$ for all $m\in \mathbb N$. 
In this section we will fix an natural number $n$. Then $S_{n} \hookrightarrow S_{n+1}$ is square-zero thickening with ideal of definition $\mathfrak{a}_{n+1} = p^n W_{n+1}(k)$. 
Let $X_{n+1}$ be a proper smooth $S_{n+1}$-scheme with relative flat normal crossing divisor $D_{n+1}$. 
Denote  $X_n =X_{n+1}\times_{S_{n+1}}S_n$, $D_n=D_{n+1} \times_{S_{n+1}}S_n$ and the square-zero ideal sheaf $\mathcal{I}_{n+1}:= \mathcal{O}_{X_{n+1}} \cdot \mathfrak{a}_{n+1}$. Thus
\[(\w S,S,\w X,X, \w D,D,\mathfrak{a},\mathcal I):=(S_{n+1},S_{n},X_{n+1},X_{n}, D_{n+1},D_{n},\mathfrak{a}_{n+1},\mathcal I_{n+1}).\]
is compatible with Setup \ref{setup}. 
Denote $X_1 = X_{n+1}\times_{S_{n+1}}S_1$ and $D_1=D_{n+1}\times_{S_{n+1}}S_1$. Since $p\cdot \mathfrak a =0$, $\mathcal I$ is an $\mathcal O_{X_1}$-module. More precisely, one has an isomorphism $\mathcal I \simeq \mathcal O_{X_1}$. Thus for any vector bundle $\mathcal F$ over $X$, one has 
\begin{equation}\label{level_1}
\mathcal F_{\mathcal I}\simeq \mathcal F\mid_{X_1}.
\end{equation}

In this section, we set
\begin{equation}\label{eqn:1_periodic}
\xymatrix{  
	& (V_n,\nabla_n,\Fil_n) \ar[dr]^{\Gr}
	& \\
	(E_n,\theta_n) \ar[ur]^{C^{-1}} 
	&&  \Gr (V_n,\nabla_n,\Fil_n) \ar@/^10pt/[ll]_{\varphi}\\
}
\end{equation} 
to be a $1$-periodic Higgs-de Rham flow over $(X_n,D_n)/S_n$ where the inverse Cartier $C^{-1}$ is defined with respect to the lift $$(X_n,D_n)/S_n\subset (X_{n+1},D_{n+1})/S_{n+1}.$$
Denote
\[(V_1,\nabla_1,\Fil_1,E_1,\theta_1,\varphi_1) = (V,\nabla,\Fil,E,\theta,\varphi)\mid_{X_1}\]
By (\ref{level_1}), the complex~\ref{complex_(V,nabla)} may be rewritten as a complex of sheaves over $X_1$
\begin{equation} \label{complex_(V,nabla)_W}
0 \rightarrow \mEnd(V_1)  \overset{\nabla_1^\End}{\longrightarrow} \mEnd(V_1)\otimes\Omega_{X_1/S_1}^1(\log D_1) \overset{\nabla_1^\End}{\longrightarrow}  \mEnd(V_1)\otimes\Omega_{X_1/S_1}^2(\log D_1) \overset{\nabla_1^\End}{\longrightarrow} \cdots
\end{equation}
which is just the de Rham complex of the de Rham bundle $\mEnd(V_1,\nabla_1)$ over $X_1$.

Similar the complexes~\ref{complex_(V,nabla,Fil)} and~\ref{complex_(Fil)} may be rewritten as 
\begin{equation}  \label{complex_(V,nabla,Fil)_W}
\Fil^\ell \DR(\mEnd(V_1,\nabla_1)): \qquad 0 \rightarrow \Fil^\ell \mEnd(V_1) \overset{\nabla_1^\End}{\longrightarrow} \Fil^{\ell-1}\mEnd(V_1)\otimes\Omega_{X_1/S_1}^1(\log D_1)  \overset{\nabla_1^\End}{\longrightarrow} \cdots
\end{equation}
and 
\begin{equation}  \label{complex_(Fil)_W}
\mathscr C=\frac{\DR(\mEnd(V_1,\nabla_1))}{\Fil^0\DR(\mEnd(V_1,\nabla_1))}: 
\qquad 0 \rightarrow \frac{\mEnd(V_1)}{\Fil^0 \mEnd(V_1)}  \overset{\nabla_1^\End}{\longrightarrow} \frac{\mEnd(V_1)}{\Fil^{-1} \mEnd(V_1)}\otimes\Omega_{X/S}^1(\log D_1)  \overset{\nabla_1^\End}{\longrightarrow} \cdots
\end{equation}

\begin{lem}\label{E1 degenarate}
	Setup as in the beginning of the subsection and suppose $(E_n,\theta_n)$ is a logarithmic Higgs bundle on $(X_n,D_n)/S_n$ that initiates a 1-periodic Higgs-de Rham flow as in Equation \ref{eqn:1_periodic}. Then the Hodge-de Rham spectral sequence associated to the reduction modulo $p$:
	\[ E^{p,q}_1 := \mathbb H^{p+q}\big(\Gr^p\DR(\mEnd(E_1,\theta_1))\big)\Rightarrow  \mathbb H^{p+q}\big(\DR(\mEnd(V_1,\nabla_1))\big) \]
	 degenerates at $E_1$.
\end{lem} 
\begin{proof} First of all, the terms of $E_\infty$ are subquotients of that of $E_1$; therefore one has 
		$$\mathrm{rank}_k\Big( \mathbb H^{p+q}(\DR(\mEnd(E_1,\theta_1))) \Big) \geq \mathrm{rank}_k \Big( \mathbb H^{p+q}(\DR(\mEnd(V_1,\nabla_1))) \Big),$$
		and moreover the spectral sequence is degenerate at $E_1$ if and only if equality holds.	
		On the other hand, by Ogus-Vologodsky,  $(V,\nabla) = C^{-1}(E,\theta)$ induces an isomorphism 
		\[ \sigma^*\mathbb H^{p+q}(\DR(\mEnd(E_1,\theta_1))) \simeq \mathbb H^{p+q}(\DR(\mEnd(V_1,\nabla_1))).\]
		In particular, one has $\mathrm{rank}_k\Big(\mathbb H^{p+q}\big(\DR(\mEnd(E_1,\theta_1))\big)\Big) = \mathrm{rank}_k \Big( \mathbb H^{p+q}(\DR(\mEnd(V_1,\nabla_1))) \Big)$ and $E_1$ degeneration.

\end{proof}

\begin{thm}\label{uniqueHDF} Let $(X,D)/W$ be a smooth pair with $X/W$ projective. Let $HDF_{X_n}$ be a periodic Higgs-de Rham flow over $(X_n,D_n)$. For a given lifting $(E,\theta)_{X_{n+1}}$ of the initial graded Higgs bundle over $(X_{n+1},D_{n+1})$, there is at most one Higgs-de Rham flow with initial term  $(E,\theta)_{X_{n+1}}$ that lifts $HDF_{X_n}$ up to isomorphism. 
\end{thm}
\begin{proof} Suppose there exists a Higgs-de Rham flow lifting $HDF_{X_n}$ with initial Higgs term $(E,\theta)_{n+1}$.
\begin{equation*} \tiny
\xymatrix@W=2cm@C=-8mm{
	&&\Big(V,\nabla,\Fil\Big)_{n+1}
	\ar[dr]^{\mathrm{Gr}}\ar@{.>}[ddd]
	&&\Big(V,\nabla,\Fil\Big)'_{n+1}
	\ar[dr]^{\mathrm{Gr}}\ar@{.>}[ddd]
	&& \cdots
	\ar[dr]^{\mathrm{Gr}}\ar@{.>}[ddd]
	\\
	&\left(E,\theta\right)_{n+1}
	\ar[ur]_{\mathcal C^{-1}_{n+1}}\ar@{.>}[ddd]_(0.3){\mod p^{n}}
	&&\left(E,\theta\right)_{n+1}'
	\ar[ur]_{\mathcal C^{-1}_{n+1}}\ar@{.>}[ddd]
	&&\left(E,\theta\right)_{n+1}'' 
	\ar[ur]_{\mathcal C^{-1}_{n+1}}\ar@{.>}[ddd]
	&& \cdots
	\ar@{.>}[ddd] 
	\\
	&&&&&&&\\
	\Big(V,\nabla,\Fil\Big)_{n}^{(f-1)}
	\ar[dr]^{\mathrm{Gr}}
	&&\Big(V,\nabla,\Fil\Big)_n
	\ar[dr]^{\mathrm{Gr}}
	&&\cdots
	\ar[dr]^{\mathrm{Gr}} 
	&&\Big(V,\nabla,\Fil\Big)_n^{(f-1)}
	\ar[dr]^{\mathrm{Gr}} 
	\\
	&\quad\left(E,\theta\right)_{n} 
	\ar[ur]_{\mathcal C^{-1}_{n}}
	&&\left(E,\theta\right)_{n}' 
	\ar[ur]_{\mathcal C^{-1}_{n}}
	&&\cdots
	\ar[ur]_{\mathcal C^{-1}_{n}}
	&&\left(E,\theta\right)_{n}^{(f-1)}
	\\}
\end{equation*} As we have fixed a lifting of the initial Higgs bundle, the freedom of choice of a lifting Higgs-de Rham flow are the choices of the Hodge filtrations $\Fil_{n+1},\Fil_{n+1}',\Fil_{n+1}'',\cdots$. Applying Theorem \ref{thm: corollary of E1 degenarate}(2) together with Lemma~\ref{E1 degenarate}, the isomorphism class of this filtered de Rham bundle $(V_{n+1},\nabla_{n+1},\Fil_{n+1})$ \emph{does not depend} on the choice of $\Fil_{n+1}$. We therefore see that the associated graded Higgs bundle
\[(E_{n+1},\theta_{n+1})':=\Gr_{\Fil_{n+1}}(V_{n+1},\nabla_{n+1})\] 
also and does not depend on the choice of the Hodge filtration $\Fil_{n+1}$. Inductively, it follows that each term in the lifted Higgs-de Rham flow is uniquely determined by the initial Higgs-bundle. Thus the theorem follows. 
\end{proof}
\begin{rmk}
Theorem~\ref{uniqueHDF} proves that there is at most one Higgs-de Rham flow lifting $HDF_{X_n}$ after fixing the initial graded Higgs term. However, we do not know that the lifted Higgs-de Rham flow is periodic; i.e., we don't know if the periodicity map is liftable. Even when the periodicity map is liftable, there may exist more than one lift. Therefore Theorem \ref{uniqueHDF} does not imply that there is at most one lift as a \emph{periodic} Higgs-de Rham flow. However, if the initial Higgs term is stable modulo $p$, then it easily follows that there is at most one lift, up to isomorphism, of $HDF_{X_n}$ as a \emph{periodic Higgs-de Rham flow}.
\end{rmk}

\subsection{The torsor map induced by the inverse Cartier transform}
Recall that
\[
\mathrm{Def}_{(V_{n},\nabla_{n})}(X_{n+1},D_{n+1}) 
= \left\{
\begin{array}{l}
\text{de Rham bundle } (\w V_{n+1},\w \nabla_{n+1}) \text{ over }  (X_{n+1},D_{n+1})/ S_{n+1}\\
\text{ with }  (\w V_{n+1},\w \nabla_{n+1})\mid_{X_n} = (V_n,\nabla_n)\\
\end{array} 
\right\}/\sim\]                                                   
and 
\[\mathrm{Def}_{(V_n,\nabla_n,\Fil_n)}(X_{n+1},D_{n+1}) = 
\left\{
\begin{array}{l}
\text{filtered de Rham bundle } (\w V_{n+1},\w \nabla_{n+1},\w \Fil_{n+1}) \text{ over } (X_{n+1},D_{n+1})/ S_{n_1}\\
\text{with }  (\w V_{n+1},\w \nabla_{n+1},\w \Fil_{n+1})\mid_{X_{n}} = (V_n,\nabla_n,\Fil_n) \\
\end{array} 
 \right\}/\sim.\] 
Similarly, denote 
\[\mathrm{Def}_{(E_n,\theta_n)}(X_{n+1},D_{n+1})
 = \left\{ \begin{array}{l}
\text{graded Higgs bundle $(\w E_{n+1},\w \theta_{n+1})$ over $(X_{n+1},D_{n+1})/\w S_{n+1}$}\\
\text{ with }  (\w E_{n+1},\w \theta_{n+1})\mid_{X_n} = (E_n,\theta_n)\\
\end{array}\right\}/\sim\]  
and 
\[\mathrm{Def}_{(X_{n+1},D_{n+1})}(S_{n+2})
 = \left\{  \begin{array}{l}
 \text{smooth log pair $(X_{n+2},D_{n+2})$ over $S_{n+2}$}\\
 \text{with } (X_{n+2},D_{n+2})\times_{S_{n+2}} S_{n+1} = (X_{n+1},D_{n+1}) 
\end{array} \right\}/\sim.\]

By Theorem~\ref{deformation_(E,theta,Gr)}, $\mathrm{Def}_{(E_{n},\theta_n)}(X_{n+1},D_{n+1})$ is an $\mathbb H^1\big(\mathrm{Gr}^0 \DR(\mEnd(E_1,\theta_1))\big)$-torsor if  $\mathrm{Def}_{(E_n,\theta_n)}(X_{n+1},D_{n+1})$ is not empty. By Theorem~\ref{deformation_(V,nabla)}, $\mathrm{Def}_{(V_n,\nabla_n)}(X_{n+1},D_{n+1})$ is an $\mathbb H^1\big( \DR(\mEnd(V_1,\nabla_1))\big)$-torsor  if  $\mathrm{Def}_{(V,\nabla)}(\w X,\w S)$ is not empty. It is well-known that $\mathrm{Def}_{(X_{n+1},D_{n+1})}(S_{n+2})$ is an $H^1(X_1,\mathcal T_{(X_1,D_1)/S_1})$-torsor if it is non-empty, see\cite[Prop. 3.14 (iii)]{Kato88}. Here, $T_{(X_1,D_1)/S_1}$ refers to the \emph{logarithmic Tangent bundle}.

\paragraph{\emph{The semilinear map $\alpha$.}} Denote by $F_{S_1}\colon S_1\rightarrow S_1$ the absolute Frobenius on $S_1$. Denote $(E_1',\theta_1')$ be the pull back Higgs bundle on $X_1'$ via the Frobenius map $F_{S_1}\colon X_1'\rightarrow X_1$. Then $(E_1',\theta_1')$ is the image of $(V_1,\nabla_1)$ under the Cartier functor constructed in Ogus-Vodogodsky, since $(V_1,\nabla_1)=C_{X_1\subset X_2}^{-1}(E_1,\theta_1)$. Thus by \cite[Corollary 2.27]{OgVo07}, one has an isomorphism 
\[\mathbb H^1\big(\DR(\mEnd(E_1',\theta_1'))\big) \rightarrow \mathbb H^1\big( \DR(\mEnd(V_1,\nabla_1))\big).\]
Composing with the Frobenius map $F_{S_1}^*\colon \mathbb H^1\big(\DR(\mEnd(E_1,\theta_1))\big) \rightarrow \mathbb H^1\big(\DR(\mEnd(E_1',\theta_1')\big)$, one gets a $\sigma$-semilinear map
\[\alpha \colon \mathbb H^1\big(\DR(\mEnd(E_1,\theta_1))\big) \rightarrow \mathbb H^1\big( \DR(\mEnd(V_1,\nabla_1))\big).\]
Since $(E,\theta)$ is a graded Higgs bundle, the $\mathbb H^1\big(\Gr^0\DR(\mEnd(E_1,\theta_1)\big)$ is a direct summand of $\mathbb H^1\big(\DR(\mEnd(E_1,\theta_1)\big)$. By abuse of notation, we still denote by $\alpha$ the restriction of $\alpha$ on this direct summand.

\paragraph{\emph{The semilinear map $\beta$.}} Let's recall the a map constructed in~\cite[Section 5]{LSYZ14}. Write the Higgs field $\theta_1$ of $E_1$ in the following way
\[\theta_1\colon (\mathcal T_{X_1/S_1},0) \rightarrow \mEnd(E_1,\theta_1).\]
Applying the inverse Cartier transform to $\theta_1$, one gets the corresponding morphism
\[C_{X_1\subset X_2}^{-1}(\theta_1)\colon (F^*\mathcal T_{X_1/S_1},\nabla_{\rm can}) \rightarrow \mEnd(V_1,\nabla_1).\]
Composing with the Frobenius map $F^*\colon \mathcal T_{X_1/S_1} \rightarrow F^*\mathcal T_{X_1/S_1}$ and then taking cohomology groups, one gets a map
\[\beta\colon H^1(X_1,\mathcal T_{X_1/S_1}) \longrightarrow \mathbb H^1\big( \DR(\mEnd(V_1,\nabla_1))\big)\]

\paragraph{The construction of truncated inverse Cartier transform.} Let $(\w E_{n+1},\w \theta_{n+1})$ be a graded Higgs bundle over $X_{n+1}$ and $(V_n,\nabla_n,\Fil_n)$ be a filtered de Rham bundle together with an isomorphism
\[\psi\colon \Gr_{\Fil_n}(V_n,\nabla_n) \rightarrow (\w E_{n+1},\w \theta_{n+1}) \mid_{X_n}\]
Then the \emph{inverse Cartier transform} yields filtered de Rham bundle 
\[C^{-1}_{(X_{n+1},D_{n+1})\subseteq (X_{n+2},D_{n+2})}\big((\w E_{n+1},\w \theta_{n+1}),(V_n,\nabla_n,\Fil_n),\psi\big)\]
over $X_{n+1}$ which is constructed as following. (See also \cite[Theorem 4.1]{LSZ13a} and its proof.)
\subparagraph{In small affine cases.} Assume $X_{n+2}$ is small affine with coordinates $t_1,\cdots,t_d$  and a Frobenius lifting $\Phi_{n+2}\colon X_{n+2}\rightarrow X_{n+2}$. This means the following: $\Phi_{n+2}$ preserves the divisor $D_{n+2}$ and lifts the the absolute Frobenius on $X_1$. We also assume all vector bundles appeared are free. By the freeness, there exists a filtered vector bundle with connection (not necessarily integrable) $(V_{n+1},\nabla_{n+1},\Fil_{n+1})$ over $X_{n+1}$ which lifts $(V_{n},\nabla_{n},\Fil_{n})$ and there exists an isomorphism which lifts $\psi$
\[\Gr_{\Fil_{n+1}}(V_{n+1},\nabla_{n+1})\xrightarrow{\sim} (\w E_{n+1},\w \theta_{n+1}).\] 
Applying Faltings' tilde functor on $(V_{n+1},\nabla_{n+1},\Fil_{n+1})$ \cite[Ch. 2]{Fal89} (see also the functor $G_n$ on \cite[p. 25-26]{LSZ13a})
one gets a vector bundle $\widetilde H$ over $X_{n+1}$ with a $p$-connection $\w \nabla$. Let's note that this vector bundle with $p$-connection 
\begin{equation}\label{local p-connection}
(\w H,\w \nabla)
\end{equation}
does not depend on the choice of the lifting $(V_{n+1},\nabla_{n+1},\Fil_{n+1})$. Taking the $\Phi_{n+2}$-pullback,
one constructs a de Rham bundle over $X_{n+1}$
\[C^{-1}_{(X_{n+1},D_{n+1})\subseteq (X_{n+2},D_{n+2})}\big((\w E_{n+1},\w \theta_{n+1}),(V_n,\nabla_n,\Fil_n),\psi\big) := \left(\Phi_{n+1}^*(\w H),\frac{\Phi_{n+2}^*}{p}(\w \nabla)\right).\]
This de Rham bundle does not depend on the choice of the Frobenius lifting upto an canonical isomorphism. i.e. For any other Frobenius lifting $\Psi_{n+2}$ on $X_{n+2}$, there is an canonical isomorphism
\[G_{\Phi_{n+2},\Psi_{n+2}}\colon \left(\Phi_{n+1}^*(\w H),\frac{\Phi_{n+2}^*}{p}(\w \nabla)\right) \rightarrow \left(\Psi_{n+1}^*(\w H),\frac{\Psi_{n+2}^*}{p}(\w \nabla)\right)\]
which is defined by a Taylor formula
\[ m\otimes_{\Phi^*_{n+1}} 1\mapsto \sum_{I} \widetilde\nabla_{n+1}(\partial)^I(m)\otimes_{\Psi^*_{n+1}} \frac{\big(\Phi^*_{n+2}(t)-\Psi^*_{n+2}(t)\big)^I}{I!\cdot p^{|I|}} \]
For more details about this type of Taylor formula, see the formula in proof of Theorem 2.3 in \cite{Fal89}.
\subparagraph{In general.}  
There exists covering of small affine open subsets $\{X_{n+2,i}\}_i$ of $X_{n+2}$ together with Frobenius liftings $\Phi_{n+2,i}$ on $X_{n+2,i}$  for each $i$. By taking a fine enough covering, one may assume all vector bundles appearing are free. According the construction in small affine cases, one gets a family local de Rham bundles
\[(H_i,\nabla_i)= C^{-1}_{(X_{n+1,i},D_{n+1,i})\subseteq (X_{n+2,i},D_{n+2,i})}\Big((\w E_{n+1},\w \theta_{n+1})\mid_{X_{n+1,i}},(V_n,\nabla_n,\Fil_n)\mid_{X_{n,i}}\Big).\]
One obtains the de Rham bundle $C^{-1}_{(X_{n+1},D_{n+1})\subseteq (X_{n+2},D_{n+2})}\big((\w E_{n+1},\w \theta_{n+1}),(V_n,\nabla_n,\Fil_n),\psi\big)$ by gluing $(H_i,\nabla_i)$ via the canonical isomorphisms $G_{ij}:= G_{\Phi_{n+2,i},\Phi_{n+2,j}}$. 
\begin{rmk}
	The local vector bundles with $p$-connections over $X_{n+1,i}$ as in (\ref{local p-connection}) can be glued into a global vector bundles with $p$-connection; Lan-Sheng-Zuo~\cite{LSZ13a} denote it as 
	\begin{equation}\label{p-connection}
	 \mathcal T_n(\w E_{n+1},\w \theta_{n+1},V_{n},\nabla_n,\Fil_n,\psi).
	\end{equation} 
	Once one obtains this $p$-connection, one can use the method in Faltings~\cite{Fal89} to construct a de Rham bundle	
	$C^{-1}_{(X_{n+1},D_{n+1})\subseteq (X_{n+2},D_{n+2})}\big((\w E_{n+1},\w \theta_{n+1}),(V_n,\nabla_n,\Fil_n),\psi\big)$. Locally pull back the $p$-connection with local Frobenius maps and then glue via Taylor formula.
\end{rmk}

Following the explicit construction of the inverse Cartier functor over the truncated level, one has following result. 
\begin{lem}\label{torsor pair} Fix an element $(\w E_{n+1},\w \theta_{n+1})\in \mathrm{Def}_{(E_n,\theta_n)}(X_{n+1},D_{n+1})$ and $(X_{n+2},D_{n+2}) \in \mathrm{Def}_{(X_{n+1},D_{n+1})}(S_{n+2})$. For any $\varepsilon \in \mathbb H^1\big(\mathrm{Gr}^0 \DR(\mEnd(E_1,\theta_1))\big)$ and any $\eta \in H^1(X_1,\mathcal T_{X_1/S_1})$, one has an isomorphism of logarithmic de Rham bundles on $(X_{n},D_n)/S_n$: 
\[C^{-1}_{X_{n+1}\subset (X_{n+2}+\eta)}\big((\w E_{n+1},\w \theta_{n+1})+\varepsilon\big) = C^{-1}_{X_{n+1}\subset X_{n+2}}(\w E_{n+1},\w \theta_{n+1}) +\alpha(\varepsilon)+\beta(\eta).\]
Here, $X_{n+2}+\eta$ is a lift of $(X_{n+1},D_{n+1})$ over $S_{n+2}$, $(\w E_{n+1},\w \theta_{n+1})+\varepsilon$ is a lift of $(E_n,\theta_n)$ over $(X_{n+1},D_{n+1})$, and $C^{-1}_{X_{n+1}\subset X_{n+2}}(\w E_{n+1},\w \theta_{n+1}) +\alpha(\varepsilon)+\beta(\eta)$ is a lift of $(V_{n},\nabla_n)$ over $(X_{n+1},D_{n+1})$.

In other words, the following diagram commutes 
	\begin{equation}
	\xymatrix{
		\mathrm{Def}_{(E_{n},\theta_{n})}(X_{n+1},D_{n+1}) \times \mathrm{Def}_{(X_{n+1},D_{n+1})}(S_{n+2})  
		\ar[d]^{1:1}\ar[rr]^(0.6){\mathrm{IC}}
		&  &\mathrm{Def}_{(V_n,\nabla_n)}(X_{n+1},D_{n+1}) 
		\ar[d]_{1:1}\\
		\mathbb H^1\big(\mathrm{Gr}^0 \DR(\mEnd(E_1,\theta_1))\big) \oplus H^1(X_1,T_{X_1}) \ar[rr]^(0.6){(\alpha,\beta)} && \mathbb H^1\big( \DR(\mEnd(V_1,\nabla_1))\big)\\
	}
	\end{equation}
	where the top horizontal arrow is the (parametrized) inverse Cartier transform $IC$, the left vertical arrow is given by the formula:
	
	$$((\w E_{n+1},\w \theta_{n+1})',(X_{n+2},D_{n+2})')
	\mapsto 
	(b((\w E_{n+1},\w \theta_{n+1}),(\w E_{n+1},\w \theta_{n+1})'), b((X_{n+2},D_{n+2}),(X_{n+2},D_{n+2})')),$$
	the right vertical arrow is given by the formula
	$$(\w V_{n+1},\w \nabla_{n+1})' \mapsto b((\w V_{n+1},\w \nabla_{n+1}),(\w V_{n+1},\w \nabla_{n+1})'),$$
	where 		
	$$(\w V_{n+1},\w \nabla_{n+1}) = C^{-1}_{(X_{n+1},D_{n+1})\subset (X_{n+2},D_{n+2})}(\w E_{n+1},\w \theta_{n+1}),$$ and $$(\w V_{n+1},\w \nabla_{n+1})' = C^{-1}_{(X_{n+1},D_{n+1})\subset (X_{n+2},D_{n+2})'}((\w E_{n+1},\w \theta_{n+1})')$$.	
\end{lem}

\begin{proof} 
Set the following notation:
\begin{equation*}	
\begin{split}
 & (H,\nabla) := C^{-1}_{X_{n+1}\subset X_{n+2}}(\w E_{n+1},\w \theta_{n+1});\\		
& (H,\nabla)' :=  C^{-1}_{X_{n+1}\subset (X_{n+2}+\eta)}\big((\w E_{n+1},\w \theta_{n+1})\big); \\
&  (H,\nabla)'':=C^{-1}_{X_{n+1}\subset (X_{n+2})}\big((\w E_{n+1},\w \theta_{n+1})+\varepsilon\big).\\	
\end{split}
\end{equation*}
Then we need to show 
\begin{equation}\label{part1}
(H,\nabla)' = (H,\nabla)+\beta(\eta),
\end{equation}
and 
\begin{equation}\label{part2}
(H,\nabla)'' = (H,\nabla)+\alpha(\varepsilon).
\end{equation} 
By definition, $(H,\nabla)$ is the de Rham bundle glued from local de Rham bundles
\[(H_i,\nabla_i)= \left(\Phi_{n+1,i}^*(\w H),\frac{\Phi_{n+2,i}^*}{p}(\w \nabla)\right)\]
via local isomorphisms $G_{ij}\colon (H_i,\nabla_i)\rightarrow (H_j,\nabla_j)$ over $X_{n+1,i}\cap X_{n+1,j}$ with 
\begin{equation}
G_{ij}(m\otimes_{\Phi^*_{n+1,i}}1) = \sum_{I} \widetilde\nabla_{n+1}(\partial)^I(m)\otimes_{\Phi^*_{n+1,j}} \frac{\Big(\Phi^*_{n+2,i}(t)-\Phi^*_{n+2,j} (t) \Big)^I}{I!\cdot p^{|I|}}\
\end{equation}
\paragraph{(\ref{part1}).}. Let $(\eta_{ij})_{ij}$ be a $1$-cocycle  representing the class $\eta$. Then $X_{n+2}':=X_{n+2} +\eta$ is the scheme constructed by gluing the local schemes $X_{n+2,i}$ via local isomorphisms $g_{ij}\colon X_{n+2,i}\cap X_{n+2,j} \rightarrow X_{n+2,i}\cap X_{n+2,j}$, which are the unique isomorphism
determined by 
\[g_{ij}^*(t_k) = t_k + p^{n+1}\cdot \eta_{ij}(\mathrm{d}t_{k})\]
where $t_1,\cdots,t_d$ is a system local parameters on $ X_{n+2,i}\cap X_{n+2,j}$. Then $(H,\nabla)'$ is the de Rham bundle glued from $(H_i,\nabla_i)$ via local isomorphisms $G'_{ij}\colon (H_i,\nabla_i)\rightarrow (H_j,\nabla_j)$ with 
\begin{equation}
G'_{ij}(m\otimes_{\Phi^*_{n+1,i}}1) = \sum_{I} \widetilde\nabla_{n+1}(\partial)^I(m)\otimes_{\Phi^*_{n+1,j}} \frac{\Big[\Phi^*_{n+2,i}(t)-\Psi^*_{n+2,j}\big((g_{ij}^{-1})^*(t)\big)\Big]^I}{I!\cdot p^{|I|}}\
\end{equation}
By definition, $b((V,\nabla),(V,\nabla)') =((G_{ij}'-G_{ij})_{ij},(0)_{i})$. It is a representative of the class $\beta(\eta)$ by explicit computation. Thus (\ref{part1}) holds. 
\paragraph{(\ref{part2}).} According the construction of $\mathcal T$ in \ref{p-connection}, one has   
\[(\widetilde{H},\widetilde{\nabla})'' = (\widetilde{H},\widetilde{\nabla}) +\varepsilon.\] 
where 
\[(\widetilde{H},\widetilde{\nabla}) =\mathcal T_{n+1} \Big((\w E_{n+1},\w \theta_{n+1}),(V_n,\nabla_n,\Fil_n)\Big) \text{ and } (\widetilde{H},\widetilde{\nabla})'' = \mathcal T_{n+1}\Big((\w E_{n+1},\w \theta_{n+1})+\varepsilon,(V_n,\nabla_n,\Fil_n)\Big).\]
In particular, for any fix local isomorphisms $f_i \colon \w H_i\rightarrow \w H_i''$; then the class $\varepsilon$ is represented by 
\[\Big((f_j-f_i)_{ij},(\w \nabla_i''\circ f_i - f_i\circ \w\nabla_i)_{i}\Big).\]
Taking the Frobenius pullback of $f_i$, one gets local isomorphisms $\Phi_{n+1,i}^*f_i \colon \Phi_{n+1,i}^*\w H_i\rightarrow \Phi_{n+1,i}^*\w H_i''$. Thus by definition $b((V,\nabla),(V,\nabla)'')$ is represented by 
\[\Big(\big(G_{ij}''\circ\Phi_{n+1,i}^*(f_i) - \Phi_{n+1,j}^*(f_j)\circ G_{ij}\big),\big(\frac{\Phi_{n+2,i}^*}{p}(\w \nabla''_i)\circ\Phi_{n+1,i}^*(f_i) - \Phi_{n+1,j}^*(f_j)\circ \frac{\Phi_{n+2,i}^*}{p}(\w \nabla_i)\big)\Big)  \]
Then by explicit computation, one shows that it is a represetation of the class $\alpha(\varepsilon)$. Thus 
\[(V,\nabla)'' = (V,\nabla) + \alpha(\varepsilon).\qedhere\] 
\end{proof}

\subsection{The condition of being \emph{ordinary}}\label{subsection:ordinary}

By Theorem~\ref{thm: corollary of E1 degenarate} and Lemma~\ref{E1 degenarate}, one has an injective map
$$\mathrm{Def}_{(V_n,\nabla_n,\Fil_n)}(X_{n+1},D_{n+1}) \hookrightarrow \mathrm{Def}_{(V_n,\nabla_n)}(X_{n+1},D_{n+1}).$$
The image is the kernel of $ob_{\Fil_n}$, by the definition of $ob_{\Fil_n}$. Set  
\[\mathbb K = \mathbb K_k = \mathrm{IC}^{-1}(\mathrm{Def}_{(V_{n},\nabla_n,\Fil_n)}(X_{n+1},D_{n+1}))= \ker(ob_{\Fil_n}\circ \mathrm{IC}).\]
In other words, $\mathbb K$ consists of pairs of a logarithmic Higgs bundle $(\w E_{n+1},\w \theta_{n+1})$ on $(X_{n+1},D_{n+1})/_{S_{n+1}}$ lifting $(E_n,\theta_n)$ and a lift $(\hat{X}_{n+2},\hat{D}_{n+2})/S_{n+2}$ of the smooth pair $(X_{n+1},D_{n+1}/S_{n+1}$ such that the inverse Cartier of $(\w E_{n+1},\w \theta_{n+1})$ with respect to the thickening $(X_{n+1},D_{n+1})\subset (X_{n+2},D_{n+2})$ is isomorphic to a de Rham bundle $(V_{n+1},\nabla_{n+1})$ on which $\Fil_{n}$ lifts. Also set
\[H = (\alpha,\beta)^{-1}(\mathbb H^1\big( \Fil^0 \DR(\mEnd(V_1,\nabla_1))\big)).\]

To aid the reader, we briefly recall what this is. Lemma~\ref{E1 degenarate}, one has an short exact sequence
\[0
\rightarrow \mathbb H^1\Big(\Fil^0 \DR(\mEnd(V_1,\nabla_1))\Big) 
\overset\iota\longrightarrow \mathbb H^1\Big(\DR(\mEnd(V_1,\nabla_1)) \Big)
\overset\pi\longrightarrow \mathbb H^1\Big(\mathscr C \Big) \rightarrow 0.\]
Thus $H$ is just the kernel of the $\sigma$-semilinear map $\pi\circ (\alpha,\beta)$
\[H = \ker(\pi\circ (\alpha,\beta)).\]
In partcular, $H$ is a $k$ vector subspace of 	$\mathbb H^1\big(\mathrm{Gr}^0 \DR(\mEnd(E_1,\theta_1))\big) \oplus H^1(X_1,\mathcal T_{X_1})$.

By lemma~\ref{torsor pair}, if $\mathbb K$ is not empty then it is an $H$-torsor.
\begin{equation}
\xymatrix{
	\mathbb K  \ar@{^{(}->}[d] \ar[r]^-{IC}  
	&  
	\mathrm{Def}_{(V_n,\nabla_n,\Fil_n)}(X_{n+1},D_{n+1}) 
	\ar@{^{(}->}[d] \ar[r]^-{\mathrm{Gr}} 
	&
	\mathrm{Def}_{(E_n,\theta_n)}(X_{n+1},D_{n+1}) 
	\\
	\mathrm{Def}_{(E_n,\theta_n)}(X_{n+1},D_{n+1}) \times \mathrm{Def}_{(X_{n+1},D_{n+1})}(S_{n+2}) \ar[r]^-{IC}
	& 
	\mathrm{Def}_{(V_n,\nabla_n)}(X_{n+1},D_{n+1}) 
	\ar[d]^{ob_{\Fil_n}}\\ 
	&
	\mathbb{H}^1(\mathscr{C}) 
	\\
}
\end{equation}

\begin{equation}
\xymatrix{
	H  \ar@{^{(}->}[d] \ar[r]^-{(\alpha,\beta)}  
	&  
	\mathbb H^1\big( \Fil^0 \DR(\mEnd(V_1,\nabla_1))\big)
	\ar@{^{(}->}[d]^{\iota} \ar[r]^-{\mathrm{Gr}} 
	&
	\mathbb H^1\big( \Gr^0 \DR(\mEnd(E_1,\theta_1))\big)
	\\
	\mathbb H^1\big(\mathrm{Gr}^0 \DR(\mEnd(E_1,\theta_1))\big) \oplus H^1(X_1,\mathcal T_{X_1}) 
	\ar[r]^(0.6){(\alpha,\beta)}  
	&
	\mathbb H^1\big( \DR(\mEnd(V_1,\nabla_1))\big) \ar[d]^{\pi}\\
	&
	\mathbb H^1(\mathscr C)\\
}
\end{equation}

\begin{thm}\label{main_thm} Suppose $\mathbb K$ is not empty and the projection $H\rightarrow \mathbb H^1\big(\mathrm{Gr}^0 \DR(\mEnd(E_1,\theta_1))\big)$ is surjective. Then there exists some finite extension $k'/k$ and $((\w E_{n+1},\w \theta_{n+1}),(X_{n+2},D_{n+2})) \in \mathbb K_{k'}$ such that
	\[\Gr\circ \mathrm{IC}((\w E_{n+1},\w \theta_{n+1}),(X_{n+2},D_{n+2})) = (\w E_{n+1},\w \theta_{n+1}).\]
	In other words, $(\w E_{n+1},\w \theta_{n+1})$ is a $1$-periodic Higgs bundle under the lifting $(X_{n+1},D_{n+1})\subset (X_{n+2},D_{n+2})$.	
\end{thm}
\begin{proof}
	Since the projection $H\rightarrow \mathbb H^1\big(\mathrm{Gr}^0 \DR(\mEnd(E_1,\theta_1))\big)$ is surjective, we may choose a linear section 
	\[(\mathrm{id},\tau)\colon \mathbb H^1\big(\mathrm{Gr}^0 \DR(\mEnd(E_1,\theta_1))\big) \rightarrow H\subset 	\mathbb H^1\big(\mathrm{Gr}^0 \DR(\mEnd(E_1,\theta_1))\big) \oplus H^1(X_1,T_{X_1}).\]
	Choose any $((\w E_{n+1},\w \theta_{n+1})',(X_{n+2},D_{n+2})') \in \mathbb K$. Denote $(\w V_{n+1},\w\nabla_{n+1},\w\Fil_{n+1})' = \mathrm{IC}((\w E_{n+1},\w \theta_{n+1})',(X_{n+2},D_{n+2})')$. We identify $(E_m,\theta_n)$ with $\Gr(V_n,\nabla_n,\Fil_n)$ via the periodicity map $\varphi$. Then $(\w E_{n+1},\w \theta_{n+1})'$ and $\Gr((\w V_{n+1},\w\nabla_{n+1},\w\Fil_{n+1}))'$ are both liftings of $(E_n,\theta_n)$. Let $\Delta$ be the element in $\mathbb H^1\big( \Gr^0 \DR(\mEnd(E_1,\theta_1))$ such that
	\[\Gr(\w V_{n+1},\w\nabla_{n+1},\w\Fil_{n+1})' = (\w E_{n+1},\w \theta_{n+1})' + \Delta.\] 
	For $\epsilon \in \mathbb H^1\big( \Gr^0 \DR(\mEnd(E_1,\theta_1))$ denote $(\w E_{n+1},\w \theta_{n+1}) = (\w E_{n+1},\w \theta_{n+1})' +\epsilon$ and $(X_{n+2},D_{n+2}) =(X_{n+2},D_{n+2})' + \tau(\epsilon)$. By the choice of $\tau$, $((\w E_{n+1},\w \theta_{n+1}),(X_{n+2},D_{n+2}))\in \mathbb K$. Denote $(\w V_{n+1},\w\nabla_{n+1},\w\Fil_{n+1}) = \mathrm{IC}((\w E_{n+1},\w \theta_{n+1}),(X_{n+2},D_{n+2}))$. Then
	\begin{equation}
	\begin{split}
	\Gr\circ C^{-1}_{(X_{n+1},D_{n+1})\subset (X_{n+2},D_{n+2})}(\w E_{n+1},\w \theta_{n+1})
	& =\Gr((\w V_{n+1},\w\nabla_{n+1},\w\Fil_{n+1})) \\
	& =  \Gr\Big((\w V_{n+1},\w\nabla_{n+1},\w\Fil_{n+1})' + \alpha(\epsilon) + \beta(\tau(\epsilon))\Big) \\
	& = \Gr\big(\w V_{n+1},\w\nabla_{n+1},\w\Fil_{n+1}\big)' + \Gr\circ\alpha(\epsilon) + \Gr\circ\beta\circ\tau(\epsilon)\\
	& = (\w E_{n+1},\w \theta_{n+1}) + \Delta -\epsilon  + \Gr\circ\alpha(\epsilon) + \Gr\circ\beta\circ\tau(\epsilon)\\
	\end{split} 
	\end{equation}
	Thus $(\w E_{n+1},\w \theta_{n+1})$ is periodic if and only if the following  equation holds.
	\begin{equation}\label{equ:Artin-Schreier}
	\Delta -\epsilon  + \Gr\circ\alpha(\epsilon) + \Gr\circ\beta\circ\tau(\epsilon)=0.
	\end{equation}
	Since $\Gr$ and $\tau$ are linear and $\alpha$, $\beta$ are $\sigma$-semilinear the equation~\ref{equ:Artin-Schreier} is of \emph{Artin-Schreier} type. Equation~\ref{equ:Artin-Schreier} is solvable after extending the base field $k$, and hence the theorem follows.
\end{proof}

\begin{lem}
	If $\pi\circ \beta\colon H^1(X_1,\mathcal T_{X_1/S_1}) \rightarrow \mathbb H^1(\mathscr C)$ is surjective, then so is the projection $H\rightarrow \mathbb H^1\big(\mathrm{Gr}^0 \DR(\mEnd(E_1,\theta_1))\big)$.
\end{lem}
\begin{proof}
	For any $\varepsilon \in \mathbb H^1\big(\mathrm{Gr}^0 \DR(\mEnd(E_1,\theta_1))\big)$, $-\pi\circ\alpha(\varepsilon) \in \mathbb H^1(\mathscr C)$. Since $\pi\circ \beta$ is surjective, there exists $\eta \in H^1(X_1,\mathcal T_{X_1/S_1})$ such that 
	\[-\pi\circ\alpha(\varepsilon) = \pi\circ\beta(\eta).\]
	Since $H =(\alpha,\beta)^{-1}(\mathbb H^1\big( \Fil^0 \DR(\mEnd(V_1,\nabla_1))\big))= \ker (\pi\circ(\alpha,\beta))$, one has $(\varepsilon,\eta)\in H$. By the arbitrary choice of $\varepsilon$, the projection $H\rightarrow \mathbb H^1\big(\mathrm{Gr}^0 \DR(\mEnd(E_1,\theta_1))\big)$ is surjective.	
\end{proof}

\end{document}